\documentclass[11pt]{article}   
\usepackage{fullpage}
\usepackage{graphicx}

\usepackage{latexsym,amsfonts,amsmath,theorem,amssymb}
\usepackage{stmaryrd,array,tabularx,bbm}
\usepackage{pstricks,graphicx}
\usepackage{hyperref}
\definecolor{refcol}{rgb}{0.1,0,0.6}
\hypersetup{colorlinks}
\hypersetup{linkcolor=refcol, citecolor=refcol}

\usepackage{a4wide}

\usepackage{color}
\usepackage{caption}
\usepackage{theorem}
\usepackage{cleveref}
\usepackage{framed}
\usepackage{blindtext}
\usepackage{float}
\usepackage{todonotes}
\usepackage{bbm}

\usepackage[caption=false]{subfig}
\usepackage{algorithm}
\usepackage{algorithmicx}
\usepackage[noend]{algpseudocode}
\usepackage{tikz}
\usetikzlibrary{arrows}
\usetikzlibrary{arrows.meta}
\usetikzlibrary{decorations.pathreplacing,angles,quotes}
\usepackage{caption}
\usepackage{scalerel}

\algrenewcommand\algorithmicrequire{\textbf{Input:}}
\algrenewcommand\algorithmicensure{\textbf{Output:}}
\newtheorem{theorem}{Theorem}[section]
\newtheorem{lemma}[theorem]{Lemma}
\newtheorem{proposition}[theorem]{Proposition}
\newtheorem{assumption}[theorem]{Assumption}

\newenvironment{proof}{\begin{trivlist}
    \item[\hskip\labelsep{\bf Proof.}]}{$\hfill\Box$\end{trivlist}}

{\theoremstyle{plain} \theorembodyfont{\rmfamily}

\newtheorem{remark}[theorem]{Remark}}

\numberwithin{equation}{section}
\numberwithin{figure}{section}
\numberwithin{table}{section}

\newcommand{\comment}[1]{}

\newcommand{\R}{{\mathbb{R}}}

\newcommand{\spann}{{\mathrm{span}}}

\newcommand{\vertiii}[1]{{\left\vert\kern-0.25ex\left\vert\kern-0.25ex\left\vert #1 
\right\vert\kern-0.25ex\right\vert\kern-0.25ex\right\vert}}

\definecolor{mh}{rgb}{0.1,0.45,0.1}

\definecolor{sw}{rgb}{0,0,1}

\usepackage{authblk}

\graphicspath{{./figures/}}

\title{On the ensemble Kalman inversion under inequality constraints}
\author[1]{Matei Hanu}
\author[2]{Simon Weissmann}
\date{\today}
\affil[1]{\normalsize
  Freie Universit\"at Berlin, Fachbereich Mathematik und Informatik\\
  14195 Berlin, Germany\\
\texttt{matei.hanu@fu-berlin.de}
}
\affil[2]{\normalsize
  Universit\"at Mannheim, Institut f\"ur Mathematik\\
  68138 Mannheim, Germany\\
\texttt{simon.weissmann@uni-mannheim.de}
}

\begin{document}

\maketitle

\begin{abstract}
 The ensemble Kalman inversion (EKI), a recently introduced optimisation method for solving inverse problems, is widely employed for the efficient and derivative-free estimation of unknown parameters. 
  Specifically in cases involving ill-posed inverse problems and high-dimensional parameter spaces,  the scheme has shown promising success.
 However, in its general form, the EKI does not take constraints into account, which are essential and often stem from physical limitations or specific requirements.
 Based on a $\log$-barrier approach, we suggest adapting the continuous-time formulation of EKI to incorporate convex inequality constraints. We underpin this adaptation with a theoretical analysis that provides lower and upper bounds on the ensemble collapse, as well as convergence to the constraint optimum for general nonlinear forward models. Finally, we showcase our results through two examples involving partial differential equations (PDEs).
\end{abstract}

\section{Introduction}
Mathematical models have been employed to describe a wide range of physical, biological, and social systems and processes, enabling the analysis and prediction of their behaviors. When applying a model to a specific system, calibration becomes crucial, aligning the model with observational data. This calibration, often referred to as inversion, serves as the foundation for various applications such as numerical weather prediction, medical image processing, and numerous machine learning methods. In the realm of inversion techniques, two prominent categories are variational/optimisation-based approaches and Bayesian/statistical approaches. This work focuses on a method that bridges these two approaches — the ensemble Kalman inversion (EKI) framework introduced in \cite{Iglesias_2013,Schillings2016}. EKI utilizes an ensemble Kalman-Bucy filter iteratively to address inverse problems.\\
Nevertheless, the fundamental version of EKI lacks the capability to integrate additional constraints on parameters. Such constraints frequently emerge in various applications due to additional insights into the system. Since the estimation of the EKI is usually not feasible, incorporating constraints into the EKI is a significant task. The subsequent discussion will specifically address the efficient integration of convex inequality constraints to the EKI. One approach to incorporate constraints to the EKI has been done in \cite{Chada2019}. Here, the authors introduce a projection method in the discrete-time EKI and derive a continuous-time limit. 
Since the resulting continuous-time limit admits discontinuities in the rhs, the authors proposed a smoothed system using a $\log$-barrier approach and derive convergence for linear forward models. 

 In the following paper, we extend the $\log$-barrier approach from \cite{Chada2019} to a broader class of convex inequality constraints. Moreover, we also incorporate Tikhonov regularization and provide a convergence analysis for nonlinear forward models.

\subsection{Literature overview}
From its very beginning \cite{Evensen2003}, the ensemble Kalman filter (EnKF) has found extensive use in both inverse problems and data assimilation scenarios. Its widespread application is attributed to its easy implementation and resilience with respect to small ensemble sizes \cite{Bergemann2009, Bergemann2010, Iglesias2014, Iglesias2016, Iglesias_2013, Li2009}. Stability has been addressed in works like \cite{Tong2015, Tong2016}, and convergence analysis, grounded in the continuous-time limit of EKI, has been developed in \cite{Bloemker2019, Bloemker2021, Bungert2021, Schillings2017, Schillings2016}. However, achieving convergence results in the parameter space often necessitates some form of regularisation. We primarily focus on Tikhonov regularisation, extensively analyzed for EKI in \cite{Tong2020}, for instance. Recent advancements include further analysis on Tikhonov regularisation for stochastic EKI and adaptive Tikhonov strategies to enhance the original variant \cite{Weissmann2022}. In the context of large ensemble sizes, a mean-field limit analysis is presented in \cite{Stuart2022, Ding2020}.\\
The addition of constraints to the EKI has been analysed more and more in recent years, for example \cite{Albers_2019,Herty_2020}. An extensive survey of existing methodologies for handling linear and nonlinear constraints in Kalman-based methods is available in \cite{Amor2018,Dan2010}. One popular method is to project the estimates of the EKI into the feasible set \cite{Kandepu2008, Wang2009}. The advantage of this method that the theory can be expanded to non-linear constraints.\\
Many of these variations find motivation by interpreting the updates in Kalman-based methods as solutions to corresponding optimisation problems. For additional insights, refer to \cite{Amor2018}.\\
For our analysis we consider preconditioned gradient methods that are based on \cite{Bertsekas2008,Bertsekas1982,Schmidt2011,Shikhman2009}.\\
Finally, we will also incorporate covariance inflation into our algorithm. While the ensemble collapse, which is the convergence of the particles to their mean value, leads to an improvement in the gradient approximation, it also leads to a degeneration of the preconditioner and therefore the EKI may get stuck in a solution that is far from global optimality. Variance inflation is a tool which allows us to control the speed of the collapse \cite{Anderson2008,Tong2015}.

\subsection{Preliminaries}
In the present paper, we consider an inverse problem of the following form
\begin{equation}\label{eqn:invprob}
    y=G(u)+\eta\,.
\end{equation}
The goal is to recover the unknown parameter $u\in X$, where $y\in \mathbb{R}^K$ denotes the observed data and $\eta\sim\mathcal{N}(0,\Gamma)$ is Gaussian additive observational noise with $\Gamma\in\mathbb{R}^{K\times K}$ symmetric positive definite. Moreover, we consider a possibly nonlinear forward map $G: X \rightarrow \mathbb{R}^K$ mapping from the parameter space $X$ to the observation space $\mathbb R^K$. Throughout the manuscript we assume that the parameter space is finite dimensional given by $X:=\mathbb{R}^{d}$. 
We follow a minimisation based approach to solve the inverse problem, where our goal is to find a minimiser of the Tikhonov regularised potential
\begin{equation} \label{eqn:regul_pot}
    \Phi^{\scaleto{\mathrm{reg}}{5pt}}(u):=\frac 12 \|G(u)-y\|^2_\Gamma+\frac{\lambda}{2} \|u\|^2_{C_0}\,.
\end{equation}
Here, $C_0\in\mathbb{R}^{d\times d}$ denotes a symmetric positive definite regularisation matrix and $\lambda>0$ is a regularisation parameter. Moreover, given a symmetric positive definite matrix $\Sigma\in\R^{K\times K}$ we define the rescaled norm $\|x\|_{\Sigma} := \langle x,\Sigma^{-1} x\rangle$, $x\in\R^K$, where $\langle \cdot, \cdot\rangle$ denotes the euclidean inner product over $\R^K$. Moreover, we will use $\|\cdot\|_F$ to denote the Frobenius norm. In the following, assume that the observation $y$, the regularisation matrix $C_0$ and the regularisation parameter $\lambda$ are given. Hence, we suppress the dependence of $\Phi^{\scaleto{\mathrm{reg}}{5pt}}$ on these quantities.

\subsection{Inverse problem under constraints}

In many practical scenarios the unknown parameter $u\in X$ is subjected to physical constraints. In what follows, we will assume that the set of feasible parameters is given as set of inequality constraints of the form  
\begin{equation} \label{eq:convex_constraints}
\Omega=\{u\in X : h_j(u) \leq 0,\ j=1,...,m\}\,,
\end{equation}
where $h_j:\mathbb R^d\to\mathbb R$ are convex functions for all $j=1,\dots,m$. Note that one important example would be the set of linear inequality constraints
\[\Omega=\{u\in X : \langle c_j,u\rangle +\delta_j \leq 0, j=1,...,m\},\]
where $c_j=\pm e_j,\ \delta_j\in\mathbb{R},\ j=1,...,m$ and $e_j$ denotes the $j$-th unit vector. Then $\delta_j$ denotes an upper and lower bound on the $j$-th component of $u$. Defining $h_j(u)=\langle c_j,u\rangle +\delta_j$ for all $j=1,...,m$ we are back in the representation \eqref{eq:convex_constraints}. 
Our goal is to solve the inverse problem \eqref{eqn:invprob} under convex-constraints $u\in\Omega$ by solving the constrained optimisation problem
\begin{equation} \label{eqn:cons_min_prob}
    \min_{u\in\Omega}\ \Phi^{\scaleto{\mathrm{reg}}{5pt}}(u).
\end{equation}
We make the following assumption on the considered 
objective function $\Phi^{\scaleto{\mathrm{reg}}{5pt}}$.
\begin{assumption} \label{assu:convex_lip}
    The functional $\Phi^{\scaleto{\mathrm{reg}}{5pt}}$ is $C^2(X,\mathbb{R}_+)$ as well as
    \begin{enumerate}
        \item $\mu$-strongly convex, i.e. there exists $\mu>0$ such that
            $$\Phi^{\scaleto{\mathrm{reg}}{5pt}}(x_1)-\Phi^{\scaleto{\mathrm{reg}}{5pt}}(x_2)\geq \langle \nabla \Phi^{\scaleto{\mathrm{reg}}{5pt}}(x_2),x_1-x_2\rangle+\frac{\mu}{2}\|x_1-x_2\|^2, \quad \text{for all}\ x_1,x_2 \in X.$$
        \item $L$-smooth, i.e. there exists $L>0$ such that the gradient $\nabla \Phi^{\scaleto{\mathrm{reg}}{5pt}}$ is global $L$-Lipschitz continuous:
        \[\|\nabla \Phi^{\scaleto{\mathrm{reg}}{5pt}}(x_1)-\nabla \Phi^{\scaleto{\mathrm{reg}}{5pt}}(x_2)\|\le L \|x_1-x_2\|\quad \text{for all}\ x_1,x_2 \in X\,.  \]
    \end{enumerate}
\end{assumption}
We note that while the above assumptions are formulated globally, our theoretical analysis requires these assumptions only locally in $\Omega$. The smoothness property implies the following useful descent condition
\[\Phi^{\scaleto{\mathrm{reg}}{5pt}}(x_1)-\Phi^{\scaleto{\mathrm{reg}}{5pt}}(x_2)\leq \langle \nabla \Phi^{\scaleto{\mathrm{reg}}{5pt}}(x_2),x_1-x_2\rangle+\frac{L}{2}\|x_1-x_2\|^2, \quad \text{for all}\ x_1,x_2 \in X\,\]
which is a standard property used to prove convergence of first order optimisation methods. Moreover, since $\Phi^{\scaleto{\mathrm{reg}}{5pt}}$ is assumed to be $\mu$-strongly convex it also satisfies the Polyak-\L ojasiewic (PL) inequality of form
\[\nu\|\nabla \Phi^{\scaleto{\mathrm{reg}}{5pt}}(x)\|^2\geq \Phi^{\scaleto{\mathrm{reg}}{5pt}}(x)-\Phi^{\scaleto{\mathrm{reg}}{5pt}}(x^\ast),\]
for some $\nu>0$ and all $x\in X$, where $x^\ast$ is the unique global minimiser of $\Phi^{\scaleto{\mathrm{reg}}{5pt}}$. We provide a detailed derivation in Lemma~\ref{lemma:pl-phib}. It is well-known that under $L$-smoothness and the above PL inequality the gradient descent method converges linearly towards the unique minimiser $x^\ast$ \cite{Karimi2016}.

\begin{remark}
    Usually the strong convexity of \eqref{eqn:regul_pot}, that we assume in Assumption~\ref{assu:convex_lip} can 
    be achieved through large enough regularisation parameter $\lambda$. However, due to this large choice of $\lambda$ the computed solution of the regularised problem might be irrelevant for the initial potential $\frac 12 \|G(u)-y\|^2_\Gamma$. There have been several discussions on the strong convexity of Tikhonov regularisation as well as other forms of regularisations for which we refer to \cite{Crane2021,Kokurin2010}.\\
    In Section~\ref{ex:psylinea}, we present a nonlinear example with Tikhonov regularisation to obtain a strongly convex potential. The example is based on a linear forward model where we add a nonlinear perturbation to the model leading to a nonlinear inverse problem. The magnitude of the nonlinear term is controlled by some $\varepsilon>0$. Then the regularisation parameter can be chosen depending on $\varepsilon$, i.e. $\lambda(\varepsilon)$, 
    such that we can keep the regularisation at a low level and still obtain a strongly-convex functional by controlling the magnitude of the nonlinear term.
\end{remark}
Note that under Assumption~\ref{assu:convex_lip} the optimisation problem \eqref{eqn:cons_min_prob} is convex and therefore any point $u^\ast\in\Omega$ that satisfies the Karuhn-Kash-Tucker (KKT) conditions, i.e. there exists $\lambda^\ast\in\mathbb{R}^m$ such that $u^\ast\in\Omega$ satisfying  
\begin{itemize}
    \item $\nabla \Phi^{\scaleto{\mathrm{reg}}{5pt}}(u^\ast)+\sum_{j=1}^m\lambda^\ast_j \nabla h_j(u^\ast)=0$,
    \item $\lambda^\ast_j h_j(u^*)\leq0$ for all $j\in\{1,...,m\}$,
    \item $\lambda^\ast_j\geq 0$ for all $j\in\{1,...,m\}$ and $ \sum_{j=1}^m \lambda^\ast_jh_j(u^*)=0$,
\end{itemize}
is the unique global minimiser of $\Phi^{\scaleto{\mathrm{reg}}{5pt}}$, also called the KKT point of \eqref{eqn:cons_min_prob}. 

\subsection{Our contribution}

In this manuscript we will apply the EKI as derivative-free optimisation method for solving \eqref{eqn:regul_pot} under convex inequality constraints on $u\in X$. The EKI for solving optimisation problems under box-constraints has been introduced in \cite{Chada2019}, where the authors provide a convergence analysis for linear forward operators without regularisation. The purpose of this work is to make use of the gradient flow structure of EKI presented in \cite{Weissmann_2022} for extending the convergence analysis of linear EKI under box-constraints to nonlinear EKI under convex inequality constraints. Moreover, our proposed scheme allows to incorporate Tikhonov regularisation. We make the following contribution:
\begin{itemize}
    \item We suggest a new adaptation of EKI, enabling the integration of convex inequality constraints on the unknown parameters using a $\log$-barrier penalty approach. The adaptation incorporates Tikhonov regularisation as well as covariance inflation. 
    \item Under strong convexity and smoothness, we provide a convergence analysis of our adaptation, where we analyse feasibility, ensemble collapse and convergence to the unique KKT-point.
    \item We demonstrate our findings through two examples based on partial differential equations (PDEs). In order to keep the implementation of the proposed scheme efficient we apply an adaptively increasing penalty parameter during the algorithm.
\end{itemize}
The paper is organised as follows. We introduce our adaptation in Section~\ref{sec:EKI}. In Section~\ref{sec:feasibity_and_ek} we quantify the ensemble collapse of our scheme as well as verify the feasibility of the computed solutions. We analyse the convergence of our scheme in Section~\ref{sec:main_res}; before presenting our numerical experiments in Section~\ref{sec_NumExp}. Finally, we summarise our work with a conclusion in Section~\ref{sec:conclusion}.


\section{Ensemble Kalman Inversion} \label{sec:EKI}
We consider the ensemble Kalman inversion (EKI) to solve the inverse problem \eqref{eqn:invprob}, where we will focus on the continuous-time limit of the scheme, cp. \cite{Schillings2017}. The framework is the following.

Firstly, we define an initial ensemble $u_0 = (u_0^{(j)})_{j =1,\dots, J}$, $ u_0^{(j)}\in X$, $j=1,\dots,J$, of size $J\ge2$. 
In the continuous-time formulation of EKI, the particle system $u_t = (u_t^{(j)})_{j =1,\dots, J}$ then moves according to the dynamical system
\begin{align} \label{EKI_basic}
    \frac{\mathrm{d} u^{(j)}_t}{\mathrm{d}t} &= - \widehat{C}^{u,G}_t \Gamma^{-1} (G(u^{(j)}_t)-y), \qquad j =1,\dots,J 
\end{align}
with initial state $u_0$. Here, we have defined the following empirical means and covariances within the particle system $u_t$, $t\ge 0$,
\begin{align*}
    \widehat{C}^{u,G}_t &:= \frac{1}{J} \sum_{j = 1}^J (u^{(j)}_t - \overline{u}_t)\otimes(G(u^{(j)}_t) - (\overline{G}(u_t)),\\
    \overline{u}_t &= \frac{1}{J}\sum_{j=1}^{J}u^{(j)}_t,\quad \overline{G}(u_t) = \frac{1}{J}\sum_{j=1}^{J}G(u^{(j)}_t)\,.
\end{align*}

Following \cite{Tong2020} we can incorporate Tikhonov regularisation into the particle system using the time evolution
\begin{align} \label{EKI_Tikho}
    \frac{\mathrm{d} u^{(j)}_t}{\mathrm{d}t} &= - \widehat{C}^{u,G}_t \Gamma^{-1} (G(u^{(j)}_t)-y)-\widehat{C}(u_t)\alpha^{1/2} C_0^{-1}u^{(j)}_t, \qquad j=1,\dots,J\,. 
\end{align}

In case of a linear forward map $G$ the EKI dynamics \eqref{EKI_Tikho} can be written as system of gradient flows
\begin{align} \label{eqn:precondflow}
    \frac{\mathrm{d} u^{(j)}_t}{\mathrm{d}t}=-\widehat{C}(u_t)\nabla \Phi^{\scaleto{\mathrm{reg}}{5pt}}(u^{(j)})\, .
\end{align}
which are coupled through an adaptive preconditioner given by the empirical covariance matrix $\widehat{C}(u_t)$. In case of a nonlinear forward operators, this representation in general only holds approximatively. Indeed, by Taylors Theorem it can be justified that $\widehat{C}^{u,G}_t \approx \widehat{C}(u_t) D G^*(u^{(j)})$, where $D G$ denotes the (Frechet) derivative of $G$ (see \cite[Lemma~4.5]{Weissmann_2022}). It follows that \eqref{EKI_Tikho} can be viewed as derivative-free approximation of the preconditioned gradient flow \eqref{eqn:precondflow} also in the nonlinear setting. To be more precise, we make the following sufficient assumptions on the nonlinear forward map  to justify the approximate gradient flow structure of EKI.
\begin{assumption}\label{ass:linear_apprx}
    The forward operator $G$ is locally Lipschitz continuous, with constant $c_{lip}>0$ and satisfies the linear approximation property
        \begin{equation}\label{eqn:approx_error}
            G(x_1)=G(x_2)+DG(x_2)(x_1-x_2)+Res(x_1,x_2)\, ,
        \end{equation}
        for all $x_1,x_2\in X$. The approximation error is bounded by
        \begin{equation}
            \|Res(x_1,x_2)\|_2\leq b_{res}\|x_1-x_2\|_2^2\,.
        \end{equation}
\end{assumption}

\subsection{Ensemble Kalman inversion under box constraints}
In the following, we will revisit the EKI Algorithm under box-constraints and modify the formulation with the goal of minimising \eqref{eqn:regul_pot} under box-constraints. We refer to \cite[Section~3]{Chada2019} for more details. In the following, we consider the constrained optimisation problem
\begin{equation}
    \min_{u\in\mathcal B}\ \Phi(u),\quad \Phi(u):= \frac12\|G(u)-y\|_\Gamma^2\, ,
\end{equation}
where $\mathcal{B}=\{u\in \R^d: a_i\leq u_i\leq b_i, i=1,...,m\}, m\leq d$, denotes the considered box. In \cite{Chada2019} the authors incorporate the box-constraints into the algorithm using am element-wise projection into the box defined as
\begin{align*}
        (\mathcal{P}(u))_i &= 
        \left\{
            \begin{aligned}
                a_i, \quad &\text{if}\ u_i<a_i,\\
                u_i, \quad &\text{if}\ u_i\in \left[a_i,b_i\right],\\
                b_i, \quad &\text{if}\ u_i>b_i,
            \end{aligned} 
        \right.\quad  i=1,...,m,\\
        (\mathcal{P}(u))_i &= u_i \quad i=m+1,...,n.
    \end{align*}
Following the idea of projected gradient methods \cite{Bertsekas1982,Shikhman2009}, EKI under box-constraints in discrete time proceeds by projecting the ensemble of particles into the box as introduced in \cite{Chada2019}.
The authors define the following variables
$$u^{(j)}_{n,p}=\mathcal{P}(u^{(j)}_n), \quad \bar{u}_{\mathcal{P}}=\frac{1}{J}\sum_{j=1}^J\left(u^{(j)}_{n,p}\right), \quad \bar{G}_{\mathcal{P}}=\frac{1}{J}\sum_{j=1}^J G\left(u^{(j)}_{n,p}\right),$$
as well as the empirical covariance matrices
\begin{align*}
    \widehat{C}^{u,u}_{n,\mathcal{P}} &:= \frac{1}{J} \sum_{j = 1}^J \left(u^{(j)}_{n,p} - \bar{u}_{\mathcal{P}}\right) \otimes \left(u^{(j)}_{n,p} - \bar{u}_{\mathcal{P}}\right),\\
    \widehat{C}^{u,G}_{n,\mathcal{P}} &:= \frac{1}{J} \sum_{j = 1}^J \left(u^{(j)}_{n,p} - \bar{u}_{\mathcal{P}}\right) \otimes \left(G\left(u^{(j)}_{n,p}\right) - \bar{G}_{\mathcal{P}}\right),\\
    \widehat{C}^{G,G}_{n,\mathcal{P}} &:= \frac{1}{J} \sum_{j = 1}^J \left(G\left(u^{(j)}_{n,p}\right) - \bar{G}_{\mathcal{P}}\right) \otimes \left(G\left(u^{(j)}_{n,p}\right) - \bar{G}_{\mathcal{P}}\right),
\end{align*}
The update formula of the EKI under box-constraints then consists of a prediction step and a projection onto the box $\mathcal{B}$.
\begin{equation} \label{eqn:discrete_update}
\begin{cases}
\tilde{u}^{(j)}_{n+1,p}&=u^{(j)}_{n,p}+\widehat{C}^{u,G}_{n,\mathcal{P}}(\widehat{C}^{G,G}_{n,\mathcal{P}}+h^{-1}\Gamma)^{-1}(y-G(u^{(j)}_{n,p})),\\
    u^{(j)}_{n+1,p}&=\mathcal{P}(\tilde{u}^{(j)}_{n+1,p}).
\end{cases}
\end{equation}
Taking the limit $h\to0$ one obtains the continuous-time formulation of EKI under box-constraints given by a coupled system of ODEs
    \begin{align}
        \left(\frac{\mathrm{d} u^{(j)}_t}{\mathrm{d}t}\right)_i &= 
            \begin{cases}
                v_i(u^{(j)}_t), & (u^{(j)}_t)_i\in \left(a_i,b_i\right),\\
                \mathbbm{1}_{\left[0,\infty\right)}(v_i(u^{(j)}_t))v_i(u^{(j)}_t), &(u^{(j)}_t)_i=a_i,\\
                \mathbbm{1}_{\left[-\infty,0\right)}(v_i(u^{(j)}_t))v_i(u^{(j)}_t), &(u^{(j)}_t)_i=b_i,
            \end{cases} 
        \quad i=1,...,m,\label{eqn:box_cont_limit}\\
        \left(\frac{\mathrm{d} u^{(j)}_t}{\mathrm{d}t}\right)_i &= u_i \quad i=m+1,...,n, \notag
    \end{align}
where 
\[v_i(u^{(j)}_t)=\left[-\widehat{C}^{u,G}_t \Gamma^{-1} (G(u^{(j)}_t)-y)\right]_i,  \]
which in the linear setting again simplifies to 
\[v_i(u^{(j)}_t)=\left[-\widehat{C}(u_t) \nabla \Phi(u_t^{(j)}) \right]_i \,. \]
Roughly speaking, the dynamical system \eqref{eqn:box_cont_limit} evolves similarly to EKI \eqref{EKI_basic}, but forces the particles to stay within the box by deactivating directions pointing outside of the box at the boundary. The resulting system of ODEs occurs with two cruical problems, which are discontinuities of the RHS of \eqref{eqn:box_cont_limit} and  the fact that straightforward preconditioning may lead to forcing directions which are no descent direction with respect to the potential $\Phi$. The latter one is a well known problem for preconditioned gradient methods \cite{Bertsekas1982}. To overcome this issue, the authors propose to consider a smoothed coupled system of ODEs given by
\begin{equation}\label{eq:EKI_box_smooth}
    \frac{\mathrm{d} u^{(j)}_t}{\mathrm{d}t}=\tau v(u^{(j)}_t)+\sum_{i=1}^{2m}\frac{1}{h_i(u_t^{(j)})}\nabla h_i(u_t^{(j)}),
\end{equation}
where $h_i(u)=a_i-u_i$ for $i=1,...,m$ and $h_{i+m}(u)=u_i-b_i$ for $i=1,...,m$. The purpose of introducing this form of smoothing is the resulting connection to the unconstrained optimisation problem
    \begin{equation}\label{eqn:log_barrier_sys_noreg}
        \min_{u\in\mathbb{R}^{d}} \tau\Phi(u)-\sum_{i=1}^{2m}\log(-h_i(u))\, .\notag
    \end{equation} 
using a $\log$-barrier penalty approach. Indeed, assuming that the forward model is linear and the constraints are described by boxes, one can prove that the solution of \eqref{eqn:box_cont_limit} solves \eqref{eqn:log_barrier_sys_noreg} in the long-time limit \cite[Theorem~3.4]{Chada2019}. In the following, we want to generalize this result to the Tikhonov regularised optimisation problem under more general convex inequality constraints.

\subsection{Tikhonov regularised EKI under convex inequality constraints } \label{sec:EKi_with_conv_BC}

Motivated by the continuous-time formulation of EKI under box constraints using the smoothed system \eqref{eq:EKI_box_smooth} we are now ready to present our considered dynamical system for solving 
\begin{equation}\label{eqn:opti_box} 
\min_{u\in\Omega}\ \Phi^{\scaleto{\mathrm{reg}}{5pt}}(u)\,,
\end{equation}
where $\Omega = \{u\in X: h_j(u)\le 0,\ j=1,\dots,m\}$ describes our set of inequality constraints. To be more precise, we make the following assumptions on the feasible set $\Omega$ for convex and continuously differentiable functions $h_j:X\to\mathbb R$.
\begin{assumption}\label{ass:inequality_constr}
    We assume that for each $j=1,\dots,m$ the function $h_j:X\to\mathbb R$ is convex and continuously differentiable. Moreover, we assume that the interior of $\Omega$ is non-empty, i.e.~there exists $u\in X$ such that $h_j(u)<0$ for each $j=1,\dots,m$.
\end{assumption}
\begin{remark}\label{rem:non_degenerate_grad}
    Note that by Assumption~\ref{ass:inequality_constr} it follows that for any $u\in X$ with $h_j(u) = 0$ we have that $\nabla h_j(u)\neq 0$. If $\nabla h_j(u) = 0$ for some $u\in X$ with $h_j(u)=0$, by convexity of $h_j$ it would follow that $\min_{u\in X}\ h_j(u) = 0$, which is in contradiction to Assumption~\ref{ass:inequality_constr}.
\end{remark}
We firstly reformulate the constrained optimisation problem as unconstrained problem using the $\log$-barrier penalty approach
\begin{equation*}
    \min_{u\in\mathbb{R}^{d}} \ \tau\Phi^{\scaleto{\mathrm{reg}}{5pt}}(u)-\sum_{i=1}^{m}\log(-h_i(u))\, ,
\end{equation*}
where $\tau>0$ is a penalty parameter. This optimisation problem can equivalently be rewritten through
\begin{equation}\label{eqn:log_barrier_sys}
    \min_{u\in\mathbb{R}^{d}} \ \Phi^{\scaleto{\mathrm{reg}}{5pt}}(u)-\frac{1}{\tau}\sum_{i=1}^{m}\log(-h_i(u))\, ,
\end{equation}
We define $\Phi^b(u)=\Phi^{\scaleto{\mathrm{reg}}{5pt}}(u)-\frac{1}{\tau}\sum_{i=1}^{m}\log(-h_i(u))$. Observe that $\Phi^b$ is strongly convex in $\Omega$ due to $\Phi^{\scaleto{\mathrm{reg}}{5pt}}$ being strongly convex and the assumption that $h_i$ are convex. Hence, for any $\tau>0$ there exists a unique global minimiser $u_\ast^\tau\in \Omega$ 
of $\Phi^b$. Note that $\Phi^b(u)$ is not well-defined for $u\notin \Omega\setminus \partial \Omega $, and hence we are working with the convention that $\Phi^b(u)=+\infty$ for $u\notin\Omega\setminus \partial \Omega$. Moreover, using duality arguments one can even bound
\begin{equation} \label{eqn:error_true_kkt}
    \Phi^{\scaleto{\mathrm{reg}}{5pt}}(u_\ast^\tau)-\Phi^{\scaleto{\mathrm{reg}}{5pt}}(u_\ast)\le \frac{m}{\tau}\, ,
\end{equation}
where $u_\ast$ denotes the unique minimiser of the constrained problem \eqref{eqn:opti_box}, we refer to \cite[Section~11.2]{BV04} for more details. 

\begin{remark}
    We can see in \eqref{eqn:error_true_kkt} that the computed solution from solving problem \eqref{eqn:log_barrier_sys} is approaching the true KKT point $u_\ast$ as we increase our penalty parameter $\tau$. However, as $\tau$ increases the computation of the minimiser of \eqref{eqn:log_barrier_sys} using an iterative optimisation method becomes more expensive. This is a well-known problem for penalty methods which is due to small steps of the applied scheme which are needed close to the boundary in order to ensure feasibility. Depending on the implemented optimization scheme, one may need to check feasibility in each iteration and successively decrease the step size. In practical implementations one usually solves the optmisation problem for fixed but sequentially increasing values of the penalty parameter $\tau=\tau_k$. Here, one may use the solution of the preceding experiment as initial point for the next sequence.\\
     We will further discuss how to apply this methodology in EKI in Section~\ref{sec:implem_adap_tau}. In regard to our algorithm we propose a method to use an adaptively increasing choice of $\tau$ which keeps the computational run-time low.
\end{remark}
For solving \eqref{eqn:log_barrier_sys} we incorporate the gradient descent direction of the $\log$-barrier penalty into the dynamical system of Tikhonov regularised EKI leading to the dynamical system
\begin{equation}\label{eqn:smooth_box_eki}
    \frac{\mathrm{d} {u}_t^{(j)}}{\mathrm{d}t}= \left[-\widehat{C}^{u,G}_t \Gamma^{-1} ({G}(u_t^{(j)})-y)-\widehat{C}({u}_t) C_0^{-1}{u}_t^{(j)}\right]+\frac{1}{\tau}\widehat{C}(u_t)\sum_{i=1}^{m}\frac{1}{h_i(\bar{u}_t)}\nabla h_i(\bar{u}_t)\,,
\end{equation}
for $\tau>0$ and $j=1,\dots,J$. Let us emphasize that our considered dynamical system differs from \eqref{eq:EKI_box_smooth} in more than just the additional regularisation term. Firstly, we are incorporating the $\log$-barrier function with respect to $\bar u_t$ instead of for each particle $u_t^{(j)}$. This leads structural advantage when theoretically analysing the dynamical behavior of the ensemble mean. As result, the feasibility of the scheme with respect to the convex inequality constraints can only be guaranteed for the ensemble mean. However, in practical scenarios it may be necessary to impose feasibility for each particle which could be guaranteed by the dynamical system
\begin{equation*}
    \frac{\mathrm{d} {u}_t^{(j)}}{\mathrm{d}t}=\left[-\widehat{C}^{u,G}_t \Gamma^{-1} ({G}(u_t^{(j)})-y)-\widehat{C}({u}_t) C_0^{-1}{u}_t^{(j)}\right]+\frac{1}{\tau}\widehat{C}(u_t)\sum_{i=1}^{m}\frac{1}{h_i({u}_t^{(j)})}\nabla h_i({u}_t^{(j)})\, .
\end{equation*}
In what follows, we will focus on the dynamical system given by \eqref{eqn:smooth_box_eki}. 
As second difference to \eqref{eq:EKI_box_smooth}, we introduce a preconditioning of the gradient descent direction resulting from the $\log$-barrier term by the empirical covariance. This preconditioning ensures the well-known subspace property of ensemble Kalman methods. This means, that the particle system solving \eqref{eqn:smooth_box_eki} remains in the linear subspace spanned by the initial ensemble. We define 
\[\mathcal{S}:= u_0^\perp + \spann\{u_0^{(j)}-\bar{u}_0, j=1,...,J\}\, , \]
where $u_0^\perp=\bar{u}_0-\mathcal{P}_E \bar{u}_0$ with the projection matrix $\mathcal{P}_E=E(E^TE)^{-1}E^T$ onto $E\in\mathbb{R}^{d\times J}$ denotes matrix with columns consisting of the centered initial particles, i.e. $E=\left[u_0^{(1)}-\bar{u}_0,...,u_0^{(J)}-\bar{u}_0\right]$. Therefore, the constrained optimisation problem \eqref{eqn:opti_box} changes to 
\begin{equation} \label{eqn:cons_min_prob_subspace}
    \min_{u\in\Omega\cap\mathcal{S}} \Phi^{\scaleto{\mathrm{reg}}{5pt}}(u)\,,
\end{equation}
in case $J\le d$.

\begin{lemma}\label{lemma:subspaceproperty}
    Let $(u_0^{(j)})_{j=1,\dots,J}$ be the 
    initial ensemble. Then any solution $(u_t^{(j)},t\ge0)$ of \eqref{eqn:smooth_box_eki} remains in the affine subspace $\mathcal S$, i.e.~$u_t^{(j)}\in\mathcal S$ for all $t\ge0$ and $j=1,\dots,J$.
\end{lemma}

\begin{proof}
    The proof of the first term has been shown in \cite{Tong2020}. For the latter part we obtain
\[\frac{1}{\tau}\widehat{C}(u_t)\sum_{i=1}^{m}\frac{1}{h_i(u_t)}\nabla h_i(u_t)=\frac{1}{\tau}\frac1J \sum_{j=1}^J \sum_{i=1}^{m}\frac{1}{h_i(\bar{u}_t)} \langle u_t^{(j)}-\bar{u}_t,\nabla h_i(\bar{u}_t)\rangle (u_t^{(j)}-\bar{u}_t)\, .\]
Hence, the latter term also remains in the space spanned by the centered particles and therefore the particles remain in the affine space $\mathcal{S}$.
\end{proof}
Finally, we make the following assumption which is necessary to ensure feasibility and even the possibility for implementation. We need to ensure that the dynamical system is initialized with a feasible particle system.  
\begin{assumption}\label{ass:feasibility}
    We assume that the mean of the initial ensemble $(u_0^{(j)})_{j=1,\dots,J}$ lies in the interior of the feasible set $\Omega$, i.e.~we assume that $h_i(\bar u_0)<0$ for all  $i=1,\dots,m$. 
        Furthermore, we assume that 
        $\|u_0^{(j)}\|\leq c$ for all $j=1,...,J$ and some $c>0$.
\end{assumption}
If the initial ensemble is not feasible, we may project the particles onto $\Omega$. Hence, we assume without loss of generality that Assumption~\ref{ass:feasibility} is satisfied.

\subsection{Covariance inflation}

One fundamental key step in analysing EKI algorithms is to quantify the ensemble collapse, which is the degeneration of the spread in the particle system. While the gradient approximation improves when the particles are close to each other, the preconditioner degenerates as well and the scheme may get stuck when being far away from the optimal solution.
The dynamics of the centered particles are given by
\begin{equation}
    \frac{\mathrm{d} \bar{u}_t}{\mathrm{d}t}= \left[-\widehat{C}^{u,G}_t \Gamma^{-1} (\bar{G}(u_t)-y)-\widehat{C}({u}_t )C_0^{-1}\bar{u}_t\right]+\frac{1}{\tau}\widehat{C}(u_t)\sum_{i=1}^{m}\frac{1}{h_i(\bar{u}_t)}\nabla h_i(\bar{u}_t).\notag
\end{equation}
For motivating EKI as derivative-free optimisation method, we split the dynamical system into a (preconditioned) gradient flow and an approximation error, written as
\begin{align*}
    \frac{\mathrm{d} \bar{u}_t}{\mathrm{d}t}&= \left[-\widehat{C}({u}_t) \nabla \Phi^{\rm{reg}}(\bar u_t)\right]+\frac{1}{\tau}\widehat{C}(u_t)\sum_{i=1}^{m}\frac{1}{h_i(\bar {u}_t)}\nabla h_i(\bar{u}_t) + {\rm{Err}}(u_t^{(1)},\dots, u_t^{(J)})\\
    &= -\widehat C(u_t) \nabla \Phi^b(\bar u_t) + {\rm{Err}}(u_t^{(1)},\dots, u_t^{(J)})\,.
\end{align*}
For EKI without constraints one can indeed verify that the ensemble collapses but not too fast \cite{Weissmann_2022}. As result the approximation error ${\rm{Err}}(u_t^{(1)},\dots, u_t^{(J)})$ will degenerate in time, while the lower bound on $\widehat C(u_t)$ will ensure the convergence of the scheme. 

In order to obtain more flexibility in controlling the speed of collapse from below and above, we are going to introduce covariance inflation. In particular, we are applying the covariance inflation introduced in \cite[Section~6]{Weissmann_2022} which inflates the particle system without changing the dynamical behavior of the ensemble mean.

The centered particles of \eqref{eqn:smooth_box_eki} without covariance inflation satisfy
\begin{align*}
    \frac{\mathrm{d} (u^{(j)}_t-\bar{u}_t)}{\mathrm{d}t}= -\widehat{C}^{u,G}_t \Gamma^{-1} (G(u^{(j)}_t)-\bar{G}(u_t))-\widehat{C}({u}_t) C_0^{-1}(u^{(j)}_t-\bar{u}_t)\,,
\end{align*}
where both terms force the particles to collapse. In order to reduce the contracting forces, we wish to relax these forces by changing the dynamical system of the ensemble spread to  
\begin{align*}
    \frac{\mathrm{d} (u^{(j)}_t-\bar{u}_t)}{\mathrm{d}t}= -(1-\rho_t)\widehat{C}^{u,G}_t \Gamma^{-1} (G(u^{(j)}_t)-\bar{G}(u_t))-(1-\beta_t)\widehat{C}({u}_t) C_0^{-1}(u^{(j)}_t-\bar{u}_t)\,,
\end{align*}
for $\rho_t,\beta_t\in[0,1)$ scaling the inflation strength. However, we aim to introduce this inflation without changing the dynamical behavior of the ensemble mean itself, which is used as optimisation scheme. One possible way of achieving the covariance inflation without changing the dynamical behavior of the ensemble mean is to consider the particle evolution of form
\begin{align}\label{eqn:dyn_varinfl}
    \frac{\mathrm{d} u^{(j)}_t}{\mathrm{d}t}= p_{\rho,\beta}(u_t^{(j)})
    +\frac{1}{\tau}\widehat{C}(u_t)\sum_{i=1}^{m}\frac{1}{h_i(\bar{u}_t)}\nabla h_i(\bar{u}_t),
\end{align}
where
    \begin{equation}\label{eq:cov_force}
    \begin{split}
    p_{\rho,\beta}(u_t^{(j)})&=-\widehat{C}^{u,G}_t \Gamma^{-1} ({G}(u_t^{(j)})-y)-\widehat{C}({u}_t) C_0^{-1}{u}^{(j)}_t\\ &\quad + \rho_t\widehat{C}^{u,G}_t \Gamma^{-1} ({G}(u_t^{(j)})-\bar G(u_t)) + \beta_t \widehat C(u_t) C_0^{-1} (u_t^{(j)}-\bar u_t) \,,
    \end{split}
    \end{equation}
with $0\leq\rho_t\le 1$ and $0\leq \beta_t\le 1$ for all $t\geq0$. This incorporation of covariance inflation may be seen as artificial. However, it has a good intuition from an optimisation point of view. While in the original EKI each particle is driven by its own direction, the formulation \eqref{eqn:dyn_varinfl} additionally moves each particle into a joint direction. This interpretation can be seen more clearly when assuming $\rho_t=\beta_t$ and rewriting \eqref{eq:cov_force} by
\begin{align*}
     p_{\rho,\beta}(u_t^{(j)}) &= -(1-\rho_t)\left[\widehat{C}^{u,G}_t \Gamma^{-1} ({G}(u_t^{(j)})-y)+\widehat{C}({u}_t) C_0^{-1}{u}^{(j)}_t\right]\\
     &\quad - \rho_t \left[\widehat{C}^{u,G}_t \Gamma^{-1} (\bar {G}(u_t)-y)+\widehat{C}({u}_t) C_0^{-1}\bar {u}_t\right]\, ,
\end{align*}
where the latter term is equal for all $j=1,\dots,J$. In our theoretical analysis, we will consider an inflation factor $\rho_t\to 1$ as $t\rightarrow\infty$ and turn of the inflation for the regularisation term through setting $\beta_t=0$. Note that the presented results straightforwardly extend to $\beta_t>0$. As stated above the evolution of the ensemble mean remains
\begin{align*}
    \frac{\mathrm{d} \bar{u}_t}{\mathrm{d}t}&= -\widehat{C}^{u,G}_t \Gamma^{-1} (\bar{G}(u_t)-y)-\widehat{C}({u}_t) C_0^{-1}\bar{u}_t+\frac{1}{\tau}\widehat{C}(u_t)\sum_{i=1}^{m}\frac{1}{h_i(\bar{u}_t)}\nabla h_i(\bar{u}_t)\\
    &= 
    v(u_t)+\frac{1}{\tau}\widehat{C}(u_t)\sum_{i=1}^{m}\frac{1}{h_i(\bar{u}_t)}\nabla h_i(\bar{u}_t)\,, 
\end{align*}
where we defined 
\begin{equation} \label{eq:drift_force}
v(u_t) := -\widehat{C}^{u,G}_t \Gamma^{-1} (\bar{G}(u_t)-y)-\widehat{C}({u}_t) C_0^{-1}\bar{u}_t\, .
\end{equation}
However, the evolution of the ensemble spread changes to 
\begin{align*}
    \frac{\mathrm{d} (u^{(j)}_t-\bar{u}_t)}{\mathrm{d}t}= -(1-\rho_t)\widehat{C}^{u,G}_t \Gamma^{-1} (G(u^{(j)}_t)-\bar{G}(u_t))-\widehat{C}(u_t) C_0^{-1}(u^{(j)}_t-\bar{u}_t).
\end{align*}
We emphasize that the solutions of the inflated flow obviously still satisfy the subspace property (see \cite{Weissmann_2022}). 
Note again, the dynamics of the centered particles with and without covariance inflation are independent of $\tau$, which will later lead to a speed of collapse independent of $\tau$.

\section{Ensemble collapse and well-posedness} \label{sec:feasibity_and_ek}
We define the centered particles as $e^{(j)}_t=u^{(j)}_t-\bar{u}_t.$ Next, we define the following Lyapunov function $V_e(t)=\frac{1}{J}\sum_{j=1}^J\frac{1}{2}\|e_t^{(j)}\|^2$ describing the deviation of the ensemble of particles from its mean and we will analyse its behaviour. To be more precise, we will prove that the ensemble collapse with rate $V_e(t) \in \mathcal O(1/t)$.

\begin{lemma}[upper bound] \label{lemma:bound_enskol}
    Let $(u_0^{(j)})_{j= 1,\dots,J}\in\Omega$ be the 
    initial particle system and let $(u_t^{(j)}, t\ge0)$, $j=1,\dots,J$, be a solution of \eqref{eqn:dyn_varinfl}. Then for all $t\ge0$ it holds true that
    \[V_e(t)\leq \frac{1}{\frac{2\sigma_{min}}{J}t+V_e(0)^{-1}}\, .\]
\end{lemma}

\begin{proof}
The outline of the proof follows similarly to \cite[Lemma~4.3]{Weissmann_2022}. The time evolution of $V_e$ is given by
\begin{align*}
    \frac{\mathrm{d} V_e(t)}{\mathrm{d}t}=&\frac{2}{J}\sum_{j=1}^J \langle u_t^{(j)}-\bar{u}_t,\frac{\mathrm{d} (u^{(j)}_t-\bar{u}_t)}{\mathrm{d}t}\rangle\\
    =&-(1-\rho_t)\frac{2}{J^2}\sum_{k,j=1}^J \langle u_t^{(j)}-\bar{u}_t,(u_t^{(k)}-\bar{u}_t)(G(u^{(k)}_t)-\bar{G}(u_t))^T\Gamma^{-1} (G(u^{(j)}_t)-\bar{G}(u_t))\rangle\\
    &-\frac{2}{J^2}\sum_{k,j=1}^J \langle u_t^{(j)}-\bar{u}_t,(u_t^{(k)}-\bar{u}_t)(u_t^{(k)}-\bar{u}_t)^T C_0^{-1}(u^{(j)}_t-\bar{u}_t)\rangle\,.  
\end{align*}
We observe that 
\[ \frac{1}{J^2}\sum_{k,j=1}^J \langle u_t^{(j)}-\bar{u}_t,(u_t^{(k)}-\bar{u}_t)(G(u^{(k)}_t)-\bar{G}(u_t))^T\Gamma^{-1} (G(u^{(j)}_t)-\bar{G}(u_t))\rangle =\|\widehat{C}^{u,G}_t\Gamma^{-1/2}\|_{F}^2 \]
and similarly 
\[ \frac{2}{J^2}\sum_{k,j=1}^J \langle u_t^{(j)}-\bar{u}_t,u_t^{(k)}-\bar{u}_t\rangle\langle u_t^{(k)}-\bar{u}_t C_0^{-1}(u^{(j)}_t-\bar{u}_t)\rangle = \|\widehat C_t(u_t) C_0^{-1/2}\|_{F}^2 \]
such that 
\begin{align*}
    \frac{\mathrm{d} V_e(t)}{\mathrm{d}t}=&-2(1-\rho_t)\|\widehat{C}^{u,G}_t\Gamma^{-1/2}\|_{F}^2-2\|\widehat C_t(u_t) C_0^{-1/2}\|_{F}^2\\
    \leq&\frac{2\sigma_{min}}{J}\left(\frac{1}{J}\sum_{j=1}^J\|u^{(j)}_t-\bar{u}_t\|^2\right)^2=\frac{2\sigma_{min}}{J} V_e(t)^2,
\end{align*}
where $\sigma_{min}$ denotes the smallest eigenvalue of $C_0$. Then the claim follows by Lyapunov theory.
\end{proof}
As discussed above, in order to prove convergence of the scheme as optimisation method we will need to lower bound the covariance matrix which is used as preconditioner. We will provide the following lower bound on the smallest eigenvalue of $(\widehat C(u_t))_{t\ge0}$ which prevents the ensemble to collapse too fast. 
\begin{lemma}[lower bound]\label{lemma:ens_col_lower}
    Let $(u_0^{(j)})_{j= 1,\dots,J}\in\Omega$ be the 
    initial particle system and let $(u_t^{(j)}, t\ge0)$, $j=1,\dots,J$, be a solution of \eqref{eqn:dyn_varinfl}.
    Moreover, we assume that
    \[\eta_0 := \min_{z\in \mathcal S, \|z\|=1}\langle z,C(u_0),z\rangle >0.\]
    Then, for each $z\in \mathcal S$ with $\|z\|=1$ there exists a $t_\circ>0$ such that.
    \begin{equation*}
        \langle z, \widehat{C}(u_t)z\rangle \ge \frac{w}{(c+w\max(k_1,0)) (t+b)}
    \end{equation*}
for all $t\ge t_\circ$. The constants are $a=\frac{(1-\rho_{t_\circ})c_{lip}^2 \lambda_{max}J}{\sigma_{min}}, b=\frac{V_e(0)^{-1}J}{2\sigma_{min}}, c=2\sigma_{max}$, $k_1 = \frac{(1-a)-\eta_{t_\circ} cb }{(1-a) b^a \eta_{t_\circ}} $ and $w=1-a>0$.
\end{lemma}
\begin{proof}   
    The outline of the proof follows similarly to \cite[Lemma~4.4]{Weissmann_2022}, where include an adaptation for allowing adaptive covariance inflation $(\rho_t,t\ge0)$. Consider the dynamics of the empirical covariance matrix
    \begin{align*}
        \frac{\mathrm{d} \widehat{C}(u_t)}{\mathrm{d} t}&=\frac{1}{J}\sum_{j=1}^J\frac{\mathrm{d}  e_t^{(j)}}{\mathrm{d} t}\left(e_t^{(j)}\right)^T+\frac{1}{J}\sum_{j=1}^Je_t^{(j)}\left(\frac{\mathrm{d}  e_t^{(j)}}{\mathrm{d} t}\right)^T\\
        &=-2(1-\rho_t) \widehat{C}^{u,G}_t \Gamma^{-1}\widehat{C}^{G,u}_t -2 \widehat{C}(u_t)C_0^{-1}\widehat{C}(u_t)\, .
    \end{align*}
    Next let $z\in \mathcal{S}$ with $\|z\|=1$, then we note that
    \begin{equation} \label{eqn:dyn_cov_scalar}
        \langle z,\frac{\mathrm{d} \widehat{C}(u_t)}{\mathrm{d} t}z\rangle=-2(1-\rho_t)\|\widehat{C}^{G,u}_t z\|^2_{\Gamma}-2\|\widehat{C}(u_t)z\|_{C_0}\, . 
    \end{equation}
    Observe that 
    \begin{align*}
        \|\widehat C_t^{G,u} z\|_\Gamma^2 \le \|\widehat{C}^{G,u}_tz\|^2\|\Gamma^{-1}\|\le \|\widehat{C}^{G,u}_tz\|^2\lambda_{\max}\, ,
    \end{align*}
    where $\lambda_{\max}$ denotes the largest eigenvalue of $\Gamma^{-1}$. We apply Cauchy-Schwarz and Hölder's inequality to derive the following bound
    \begin{equation*}\label{eqn:estimate_crossc}
        \begin{split}
        \|\widehat{C}^{G,u}z\|^2 = \langle \widehat{C}^{G,u}z,\widehat{C}^{G,u}z\rangle &= \langle \frac1J\sum_{j=1}^J \langle u^{(j)}-\bar u,z\rangle \rangle (G(u^{(j)})-\bar G) , \frac1J\sum_{i=1}\langle u^{(i)}-\bar u,z\rangle \rangle (G(u^{(i)})-\bar G) \rangle \\
        &=\frac1{J^2} \sum_{i,j=1}^J \langle u^{(j)}-\bar u,z\rangle \langle u^{(i)}-\bar u,z\rangle \langle G(u^{(j)})-\bar G,G(u^{(i)})-\bar G\rangle\\
        &\le \frac1{J^2} \sum_{i,j=1}^J \langle u^{(j)}-\bar u,z\rangle \langle u^{(i)}-\bar u,z\rangle \| G(u^{(j)})-\bar G\| \, \|G(u^{(i)})-\bar G\|\\
        &=\left(\frac1J\sum_{j=1}^J\langle u^{(j)}-\bar u,z\rangle \| G(u^{(j)})-\bar G\| \right)^2\\
        &\le \left(\left(\frac1J\sum_{j=1}^J\langle u^{(j)}-\bar u,z\rangle^2\right)^{\frac12} \left(\frac1J\sum_{j=1}^J\| G(u^{(j)})-\bar G\|^2 \right)^{\frac12}\right)^2\,,
        \end{split}
    \end{equation*}
    which implies that
    \[\|\widehat C_t^{G,u} z\|_\Gamma^2 \le c_{lip}^2\lambda_{\max} V_e(t) \langle z, C(u) z\rangle\, ,\]
    since
    $$\langle u^{(j)}-\bar u,z\rangle^2=z^T (u^{(j)}-\bar u)(u^{(j)}-\bar u)^Tz.$$
    We denote by $\eta_t$ the smallest eigenvalue of $\widehat{C}(u_t)$ and by $\phi(t)$ the corresponding unit-norm eigenvector. We will show that $\eta_t>0$ for all $t\geq 0$. The following holds
    \begin{align*}
        0=\frac{\mathrm{d} \|\phi(t)\|^2}{\mathrm{d} t}=2\langle \phi(t), \frac{\mathrm{d} \phi(t)}{\mathrm{d} t}\rangle\, ,
    \end{align*}
    and hence,
    \begin{align*}
      \frac{\mathrm{d} \eta_t}{\mathrm{d} t}=& \frac{\mathrm{d} \langle \phi(t),\widehat{C}(u_t)\phi(t)\rangle}{\mathrm{d} t} \\
      =&\langle \Phi(t),\frac{\mathrm{d} \widehat{C}(u_t)\Phi(t)}{\mathrm{d} t}\rangle + \langle \frac{\mathrm{d} \Phi(t)}{\mathrm{d} t}, \widehat{C}(u_t)\Phi(t)\rangle\\
      =&\langle \Phi(t),\frac{\mathrm{d} \widehat{C}(u_t)}{\mathrm{d} t}\Phi(t)\rangle + 2\eta_t\langle \frac{\mathrm{d} \Phi(t)}{\mathrm{d} t}, \Phi(t)\rangle\\
      \geq& - 2 (1-\rho_t) c_{lip}^2 V_e(t)\lambda_{max} \langle \Phi(t),\widehat{C}(u_t)\Phi(t)\rangle - 2\sigma_{\max} \|\widehat{C}(u_t)\Phi(t)\|^2,\\
      =& - 2 (1-\rho_t) c_{lip}^2 V_e(t)\lambda_{max} \eta_t - 2\sigma_{\max} \eta_t^2\\
      \geq& -\frac{2(1-\rho_t) c_{lip}^2 \lambda_{max}}{\frac{2\sigma_{min}}{J}t+V_e(0)^{-1}}\eta_t-2\sigma_{\max} \eta_t^2\, ,
    \end{align*}
    where we used in the third step the symmetry of $\widehat{C}(u_t)$ and in the last Lemma~\ref{lemma:bound_enskol}.
    Hence we have an ODE of the form
    \begin{align*}
        \frac{\mathrm{d} \eta_t}{\mathrm{d} t}\geq -\frac{a_t}{t+b}\eta_t-c\eta_t^2,
    \end{align*}
    where $a_t=\frac{(1-\rho_t)c_{lip}^2 \lambda_{max}J}{\sigma_{min}}, b=\frac{V_e(0)^{-1}J}{2\sigma_{min}}, c=2\sigma_{max}$. Since $\rho_t\rightarrow 1$ as $t\rightarrow\infty$, there exists a $t_\circ$ such that $a_t\le a\in(0,1)$ for all $t\ge t_{\circ}$.
    The solution of the ODE
\begin{equation*}
    \frac{{\mathrm d}x(t)}{{\mathrm d}t} = - \frac{a}{t+b} x(t) - c x(t)^2
\end{equation*}
is given by
\begin{equation*}
    x(t) = \frac{1-a}{c(t+b) + (1-a) k_1 (t+b)^{a}},\quad x(0) = \eta_{t_{\circ}}>0\, ,
\end{equation*}
where $k_1$ is chosen such that the initial condition is satisfied (see Lemma~\ref{alem:ode_lower_bound}).
Since 
$a\in(0,1)$ we have that $x(t)>0$ for all $t\geq 0$. Hence, we can run our ODE in time and choose $t_\circ$ as our initial starting point. For such values of $a$ we have that $(t+b)^{a} \le (t+b)$ such that we obtain the lower bound
\begin{equation*}
    x(t) \ge \frac{1-a}{c(b+t) + (1-a) \max(k_1,0) (t+b)} = \frac{w}{(c+w\max(k_1,0)) (t+b)},
\end{equation*}
where $w=1-a>0$, which guarantees by comparison argument that
\begin{equation*}
    \eta_t \ge \frac{w}{(c+w\max(k_1,0)) (t+b)}
\end{equation*}
for all $t\ge t_\circ$.
\end{proof}

\begin{remark}
    We emphasize that the constants describing the upper and lower bound of the ensemble collapse as well as $t_\circ$ from Lemma~\ref{lemma:ens_col_lower} are independent of the penalty parameter $\tau$. 
\end{remark}
Our next result shows that any solution $u(t)$ of \eqref{eqn:dyn_varinfl}, remains in the feasible set $\Omega$ for all $t\geq0$ and therefore, there exists a unique global solution. We provide the proof of the statement in Appendix~\ref{app:A2}.
\begin{proposition}[existence] \label{lemma:ens_in_box}
    Let $(u_0^{(j)})_{j= 1,\dots,J}$ 
    be an initial particle system satisfying Assumption~\ref{ass:feasibility}. We assume that the feasible set is bounded, i.e.~that there exists $R>0$ such that $\Omega\subseteq \mathcal B_R(0)$. Moreover, we assume that $\|\nabla h_i(u)\| \le C(R)$ for some constant $C(R)>0$ depending on $R$ and all $u\in\mathcal B_R(0)$. Then the mean of any solution $(u_t^{(j)}, t\ge0)$ of \eqref{eqn:dyn_varinfl} remains in the feasible set $\Omega$, i.e.~$\bar u_t\in\Omega$ for all $t\ge0$. Moreover, there exists a unique solution of \eqref{eqn:dyn_varinfl}.
\end{proposition}
    Due to Proposition~\ref{lemma:ens_in_box} under Assumption~\ref{ass:feasibility} the mean of computed solutions is always feasible and bounded. Hence, we also obtain that 
        $\|u_t^{(j)}\|\leq B$
    for some $B>0$, all $t\ge0$ and $j=1,\dots,J$.
\section{Gradient flow structure and convergence analysis} \label{sec:main_res}
The dynamics of the particles given by the RHS of \eqref{EKI_Tikho} can be approximated by the preconditioned gradient flow 
\[-\widehat{C}(u_t)\nabla \Phi^{\scaleto{\mathrm{reg}}{5pt}}(u^{(j)})\, .\] This approximation is getting more precise, as the ensemble collapse is happening \protect{\cite[Lemma 4.5, Lemma 6.4]{Weissmann_2022}}. We consider in the following the approximation of \eqref{eqn:smooth_box_eki} for the mean of our particles given in \protect{\cite[Lemma~6.4]{Weissmann_2022}}.

\begin{lemma}[gradient flow approximation] \label{lemma:approx error}
    Consider the ensemble $\{u^{(j)},j=1,...,J\}$ and assume that Assumption~\ref{ass:linear_apprx} holds, then the mean of the particles satisfies
    \begin{equation}\label{eqn:mean_grad_approx}
        \| \widehat{C}^{u,G}_t \Gamma^{-1} (\bar{G}_t-y) -\widehat{C}(u_t)\nabla\Phi(\bar{u})\|\leq b_1 J \sqrt{\Phi(\bar{u})} V_e(t)^{\frac{3}{2}}\, ,
    \end{equation}
    where $b_1$ is independent of $J$.
\end{lemma}
We are now ready to formulate our main result of convergence.
\begin{theorem}[main convergence]\label{thm:main_conv}
    Suppose that Assumption~\ref{assu:convex_lip}, \ref{ass:linear_apprx}, \ref{ass:inequality_constr} and \ref{ass:feasibility} are satisfied. Let $u_\tau^*$ be the unique global minimiser of \eqref{eqn:log_barrier_sys}, let $(u_0^{(j)})_{j=1,\dots,J}\in \Omega$ be 
    the initial particle system with $J>d$ and let $(u_t^{(j)}, t\ge0)$, $j=1,\dots,J$, be a solution of \eqref{eqn:dyn_varinfl}. 
    Then there exist $c_1,c_2>0$ and $\bar t\ge0$ such that
    \[\Phi^b(\bar{u}_t)-\Phi^b(u_*^\tau)\le\left(\frac{c_1}{t+c_2}\right)^{\frac{1}{\gamma}} \, ,\]
    for all $t\ge \bar t$, where $\gamma:=c+w\max(k_1,0)>0$ is independent of $t$ and $\tau$. The constants $c$, $w$ and $k_1$ are defined in Lemma~\ref{lemma:ens_col_lower}.
    In particular, we have that 
    \[ \lim_{t\to\infty} \Phi^b(\bar{u}_t)-\Phi^b(u_*^\tau) = 0\,.\]
\end{theorem}
\begin{proof}
   Define $V(t)=\Phi^b(\bar{u}_t)-\Phi^b(u_*^\tau)$ and $Err(t)= \widehat{C}^{u,G}_t \Gamma^{-1} (\bar{G}_t-y) -\widehat{C}(u_t)\nabla\Phi(\bar{u})$, then we have
   \begin{align*}
       \frac{\mathrm{d} V(t)}{\mathrm{d} t}=&\langle \nabla \Phi^b(\bar{u}_t), \frac{\mathrm{d}\bar{u}_t}{\mathrm{d} t}\rangle\\
       =&\langle\nabla \Phi^b(\bar{u}_t), \left[-\widehat{C}^{u,G}_t \Gamma^{-1} (\bar{G}(u_t)-y)-\widehat{C}({u}_t) C_0^{-1}\bar{u}_t\right]+\frac1{\tau}\widehat{C}(u_t)\sum_{i=1}^{m}\frac{1}{h_i(\bar{u}_t)}\nabla h_i(\bar{u}_t)\rangle\\
       =&\langle\nabla \Phi^b(\bar{u}_t),-Err(t)\rangle\\
       &+\langle\nabla \Phi^b(\bar{u}_t), \widehat{C}(u_t)(-\nabla\Phi(\bar{u}_t)- C_0^{-1}\bar{u}_t+\frac{1}{\tau}\sum_{i=1}^{m}\frac{1}{h_i(\bar{u}_t)}\nabla h_i(\bar{u}_t))\rangle\\
        =&\langle\nabla \Phi^b(\bar{u}_t),-Err(t)\rangle+\langle\nabla \Phi^b(\bar{u}_t),-\widehat{C}(u_t)\nabla\Phi^b(\bar{u}_t)\rangle\, .
   \end{align*}
   The first term can be approximated by Cauchy-Schwarz and using \eqref{eqn:mean_grad_approx}
   \begin{align*}
       \langle\nabla \Phi^b(\bar{u}_t),-Err(t)\rangle\leq \|\nabla \Phi^b(\bar{u}_t)\|\|Err(t)\|&\leq \|\nabla \Phi^b(\bar{u}_t)\|b_1 J\sqrt{\Phi(\bar{u}_t)} V_e(t)^{\frac{3}{2}}\\
       &\leq \|\nabla \Phi^b(\bar{u}_t)\|b_1 J\sqrt{c_{lip}\|\bar{u}_t\|} V_e(t)^{\frac{3}{2}}\,.
   \end{align*}
   For the second term we obtain through Lemma~\ref{lemma:bound_enskol} and the PL-inequality, for which we give more details in Lemma~\ref{lemma:pl-phib},
   \begin{align*}
       -\langle\nabla \Phi^b(\bar{u}_t), \widehat{C}(u_t)\nabla\Phi^b(\bar{u}_t)\rangle \leq -\eta_t  \|\nabla \Phi^b(\bar{u}_t)\|^2\leq-2\mu\eta_t\left(\Phi^b(\bar{u}_t)-\Phi^b(u_*^\tau)\right),
   \end{align*}
   where $\eta_t>0$ denotes the smallest eigenvalue of $\widehat{C}(u_t)$.
Using Young's inequality it follows that $\|\nabla \Phi^b(\bar{u}_t)\| \le \frac12\|\nabla \Phi^b(\bar{u}_t)\|^2+\frac12$, such that we obtain
\begin{align*}
    \frac{\mathrm{d} V(t)}{\mathrm{d} t}&\leq
    -\eta_t \|\nabla \Phi^b(\bar{u}_t)\|^2 +\|\nabla \Phi^b(\bar{u}_t)\|b_1 J\sqrt{c_{lip}\|\bar{u}_t\|} V_e(t)^{\frac{3}{2}}\\
    &\leq -\frac{w}{(c+w\max(k_1,0)) (t+b)} \|\nabla \Phi^b(\bar{u}_t)\|^2\\ &\quad +\left(\frac12\|\nabla \Phi^b(\bar{u}_t)\|^2+\frac12\right)b_1 J \sqrt{c_{lip}B} \left(\frac{1}{\frac{2\sigma_{min}}{J}t+V_e(0)^{-1}}\right)^{\frac{3}{2}}
\end{align*}
for all $t\ge t_\circ$, where we have used Lemma~\ref{lemma:ens_col_lower}. For sufficiently large $\bar t\ge0$ we have that
\[\frac12\|\nabla \Phi^b(\bar{u}_t)\|^2 \left(b_1 J \sqrt{c_{lip}B} \left(\frac{1}{\frac{2\sigma_{min}}{J}t+V_e(0)^{-1}}\right)^{\frac{3}{2}} -\frac{w}{2(c+w\max(k_1,0)) (t+b)}\right) \le 0\,,  \]
for all $t\ge \bar t$, such that 
\begin{align*}
    \frac{\mathrm{d} V(t)}{\mathrm{d} t}
    &\leq -\frac{w}{2(c+w\max(k_1,0)) (t+b)} \|\nabla \Phi^b(\bar{u}_t)\|^2 +\frac12b_1 J \sqrt{c_{lip}B} \left(\frac{1}{\frac{2\sigma_{min}}{J}t+V_e(0)^{-1}}\right)^{\frac{3}{2}}\\
    &\leq -\frac{w\mu}{(c+w\max(k_1,0)) (t+b)}  V(t) +\frac12b_1 J \sqrt{c_{lip}B} \left(\frac{1}{\frac{2\sigma_{min}}{J}t+V_e(0)^{-1}}\right)^{\frac{3}{2}}
\end{align*}
for all $t\ge\max\{\bar t,t_\circ\}$. Using Gronwall's Lemma we obtain that
\[V(t) = \Phi^b(\bar u_t) - \Phi^b(u_\ast^\tau) \le \left(\frac{c_1}{t+c_2}\right)^{\frac{1}{\gamma}}\,, \]
for all $t\ge\max\{\bar t,t_\circ\}$, where $\gamma:=c+w\max(k_1,0)$ is independent of $t$ and $\tau$.
\end{proof}

\begin{remark} \label{rem:subspace_prop}
    The assumption $J>d$ in Theorem~\ref{thm:main_conv} is rather constraining, and usually not fulfilled in practical situations. However, in the case $J\leq d$ we have observed that the well-known subspace property Lemma~\ref{lemma:subspaceproperty} is satisfied. 
    In that case we can consider convergence to a KKT-point $u^\ast$ of \eqref{eqn:cons_min_prob_subspace} and the results in Theorem~\ref{thm:main_conv} can be generalized straightforwardly to this setting. We refer to Section~4.3 in \cite{Weissmann_2022} for more details.
\end{remark}

\section{Numerical Experiments} \label{sec_NumExp}
We now conduct two numerical experiments to test our methodology. In the first experiment, we aim to estimate the source term in a 1D parabolic PDE, using measurements obtained from its solution. However, this experiment is a linear inverse problem, therefore we incorporate a non-linear perturbation term into the potential \eqref{eqn:regul_pot}. Subsequently, in the second experiment, our goal is to estimate the log-diffusion coefficient in a 2D elliptic PDE once again using measurements of its solution. 

\subsection{Implementation of the penalty approach} \label{sec:implem_adap_tau}
We conclude the latter experiment with an implementation of an adaptive penalty parameter $\tau$. Let $u_\ast$ be the KKT-point of \eqref{eqn:opti_box} and $u_\ast^\tau\in \R^d$ the unique global minimiser of $\Phi^b$. Then the error \eqref{eqn:error_true_kkt} decreases as the penalty parameter $\tau$ increases.\\
We will compute the solutions of \eqref{eqn:dyn_varinfl} for several fixed penalty parameters $\tau$ and compute the errors to $u_\ast$ in the parameter space as well as the observations space. In Proposition~\ref{lemma:ens_in_box} we have verified that each solution of \eqref{eqn:dyn_varinfl} remains feasible for fixed $\tau$. However, in numerical implementations this cannot be guaranteed anymore. Depending on the chosen numerical solver one may need to validate feasibility in each discretisation step and successively decrease the applied step size.
Indeed, we observe that the ODE \eqref{eqn:dyn_varinfl} becomes more and more stiff with increasing $\tau$. Hence, our applied ODE solver has to take smaller step sizes in order to ensure feasibility and the algorithm becomes computationally slow. Taking this into account, we also consider an adaptive choice of the penalty parameter $\tau$, which is increasing over time, when solving the ODE \eqref{eqn:dyn_varinfl} in Section~\ref{sec:adap_tau}. To be more precise, we consider the dynamical system
\begin{align}\label{eqn:dyn_varinfl_adaptau}
    \frac{\mathrm{d} u^{(j)}_t}{\mathrm{d}t}= p_{\rho_t,0}(u_t^{(j)})
    +\frac{1}{\tau_t}\widehat{C}(u_t)\sum_{i=1}^{m}\frac{1}{h_i(\bar{u}_t)}\nabla h_i(\bar{u}_t),
\end{align}
with $(\tau_t)_{t\ge0}$ denoting our adaptive penalty parameter choice with $\lim_{t\to\infty}\tau_t=\infty$. We will observe in Section~\ref{sec:adap_tau} that this approach leads to significant reduction of the overall run-time.

\subsection{Pseudolinear example} \label{ex:psylinea}
We consider an inverse problem of the form \eqref{eqn:regul_pot} where $G(u)=Au+\varepsilon\left[\sin(u_1),...,\sin(u_K)\right]^T$, where $A\in\mathbb{R}^{K\times K}$ is symmetric and positive semi-definite and $\varepsilon>0$. We assume that $\sup_{u\in\Omega} \|u\|\le B$ and let $A_{\max} = \max_{i,j}|A_{i,j}|$, $y_{\max} = \max_{i}|y_i|$. For $\lambda > ((4+B)a_{\max} + y_{\max}) \varepsilon + \varepsilon^2 $, the Tikhonov regularised loss function 
    \[\Phi^{\scaleto{\mathrm{reg}}{5pt}}(u) = \frac12\|G(u)-y\|^2 + \frac{\lambda}2 \|u\|^2\]
    is strongly convex and $L$-smooth, see Lemma~\ref{lemma:ex_convex}.\\ For the linear model $A$ we consider a one-dimensional heat equation given by the differential equation
\begin{alignat*}{2}
            \frac{\partial u(x,t)}{\partial t} - \frac{\partial^2 u(t,x)}{\partial x^2}  &= f(x) \qquad  &&(t > 0, x \in (0,1))\\
            u(0,x) &= 0   \qquad     &&(x \in (0,1))\\
   u(t,0),u(t,1) &= 0     \qquad     &&(t\geq 0).
\end{alignat*}
We aim to recover the forcing $f$ given observations of the solution at specific grid points. For this we introduce the solution operator of the PDE $L=\frac{\mathrm{d}}{\mathrm{d} t}-\frac{\mathrm{d}^2}{\mathrm{d}^2 x}$ and an equidistant observation operator on $[0,1]\times\mathbb{R}_{\geq 0}$ defined as $\mathcal{O}:H_0^1([0,1],\mathbb{R})\rightarrow \mathbb{R}^K, (p(\cdot)) \mapsto \mathcal{O}(p(\cdot))=\left(p(x_1),...,p(x_K)\right)^T$. Here $p\in H_0^1([0,1],\mathbb{R})$ denotes a solution operator of the PDE. Then the forward operator is given by 
$A=\mathcal{O} \circ L^{-1}$, and the inverse problem is respectively
        $$u=Af+\eta,$$
where $\eta\sim\mathcal{N}(0,\Gamma)$, with $\Gamma=0.1^2 \mathrm{Id}_K$.\\
The forcing is assumed to follow a Gaussian distribution with zero mean and covariance given by $C_0:=C(s,t)=\sigma^2\exp{(-\frac{|s-t|^2}{L_{sc}})}$, where $(s,t)\in[0,1]\times[0,1]$. In this expression, $\sigma^2$ is set to $10$ and $L_{sc}$ is set to $0.1$. To simulate the random field, we employ a Karhunen-Loève (KL) expansion truncated after $12$ terms. Specifically, we express $f(x,\omega)$ as follows:

\[f(x,\omega)=\sum_{i=1}^8 \lambda_i^{1/2}e_i(x)\xi_i(\omega)\]
Here, $\lambda_i$ represents the largest eigenvalues of $C(s,t)$, $e_i(x)$ denotes the corresponding eigenfunctions, and $\xi_i$ are standard normal distributed random variables. Furthermore we draw our initial ensemble from the same random field by making $J=10$ independent draws from 
$f(x,\omega)$. Note that the generated solution of the EKI will lie in $f_0^\perp +\mathcal E$, where $\mathcal E$ is the linear span of the centered initial ensemble and $f_0^\perp=\bar f(0)-P_{\mathcal E} \bar f(0)$.
For the simulation, we use a spatial step size of $h = 0.01$, hence we have $99$ interior points and therefore the state space dimension is $d=99$, a time step size of $\Delta = 0.05$, and only solved for one step size into the future.\\
We draw the initial forcing $f^\dagger$ from $f(x,\omega)$ and compute our observations by solving the PDE given this forcing.\\ We impose convex constraints (CC) in the form of an upper bound on the norm of our solution, i.e. we set $h(u)=\frac{1}{2}\|u\|^2_{C_0}-\frac{1}{2}\|f^\dagger\|_{C_0}$ and require $h(u^{(j)}_t)\le 0$ for all $j=1,...,J$ and $t\geq 0$. \\

We compare our solution to the global optimum of \eqref{eqn:log_barrier_sys}, where we use, MATLABs optimisation tool \verb+fmincon+. We solve the ODE \eqref{eqn:dyn_varinfl} until $T=10^6$ and use \verb+ode45+ as solver. We use $\tau=10^4$ and a regularisation factor of $\beta=0.01$.\\
Furthermore, we consider two different types of covariance inflation
\begin{itemize}
    \item Increasing covariance inflation $\rho_t=1-\frac{1}{\log(t+\exp(1)}$.
    \item Constant covariance inflation $\rho_t=0.8$ for all $t\geq0$.
\end{itemize}

\begin{figure}[t]
\centering
\captionsetup{width=.9\linewidth}
\includegraphics[scale=0.45]{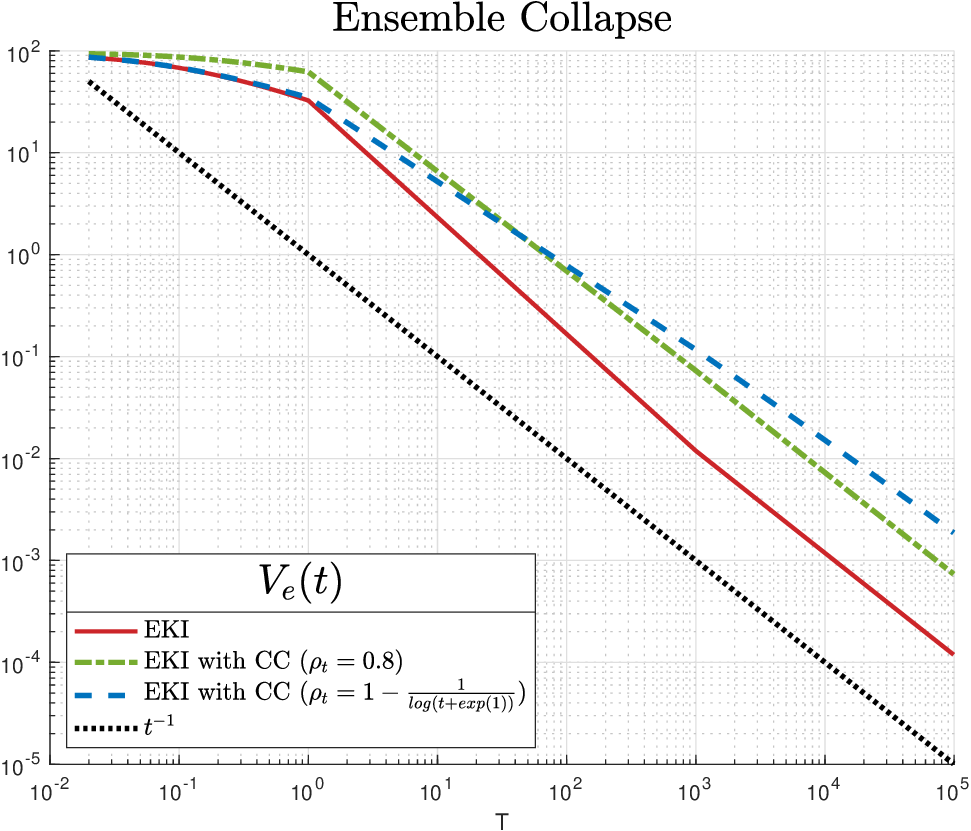}
\caption{Ensemble collapse in the form of $V_e(t)$ of the particles for the EKI (red) and EKI with BC (green line represents fixed covariance inflation; blue line increasing covariance inflation). The black dotted line depicts the rate $t^{-1}$}
\label{fig:ek_psylinear}
\end{figure}

Figure~\ref{fig:ek_psylinear} depicts the ensmble collapse of the particles in the form $V_e(t)=\frac{1}{J}\sum_{j=1}^J\frac{1}{2}\|e_t^{(j)}\|^2$. We can see that the ensemble collapse for both methods with convex constraints happen with the same rate $t^{-1}$ as as imposed by Lemma~\ref{lemma:bound_enskol}. Both methods perform similarly well and only differ from the EKI due to constants.

\begin{figure}[t]
\centering
\captionsetup{width=.9\linewidth}
\includegraphics[scale=0.4]{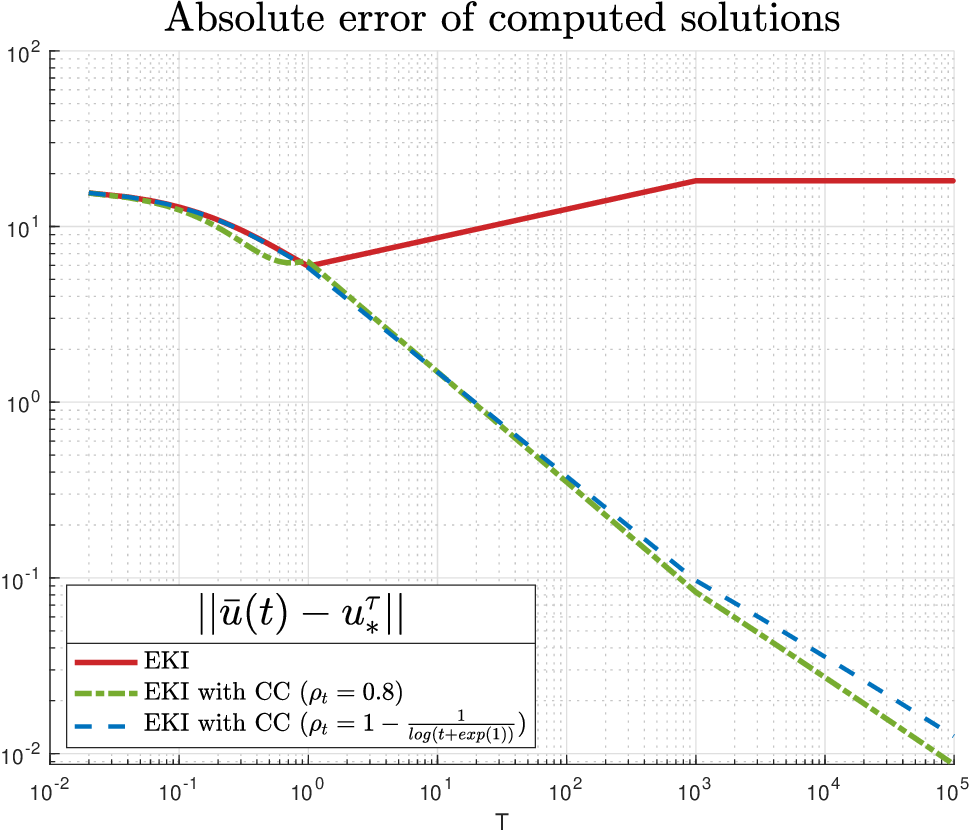}
\includegraphics[scale=0.4]{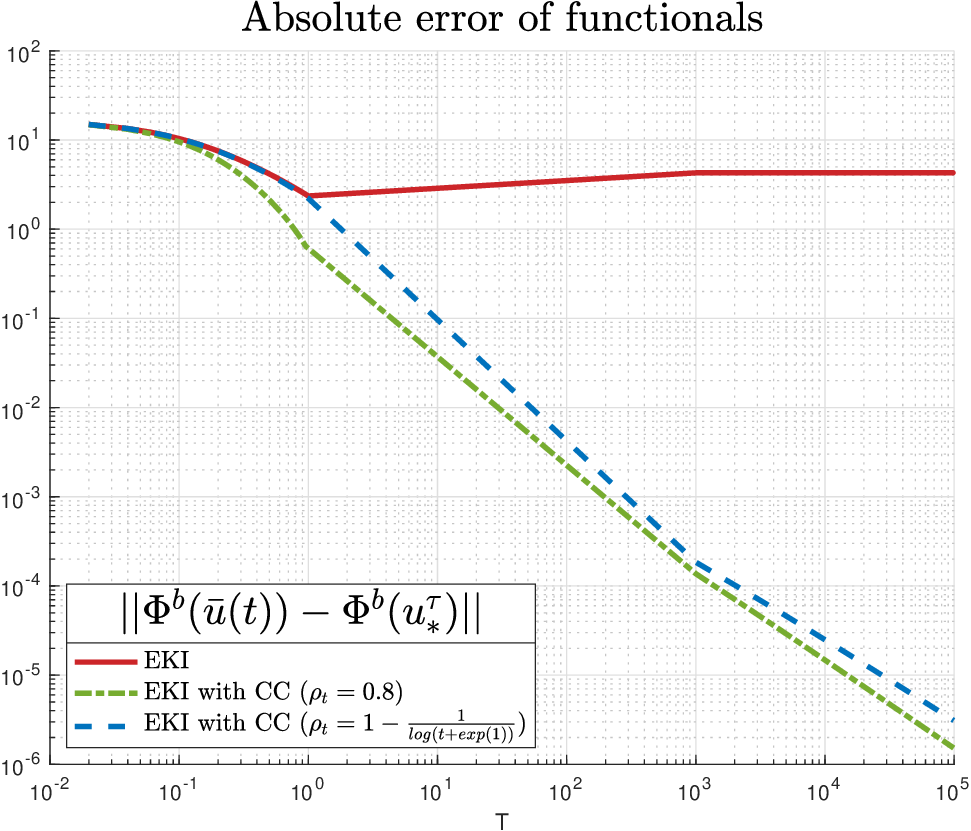}
\caption{Absolute error of computed solutions in the parameter (left) and observation space (right). The red line illustrates the EKI, the green dash-dotted line EKI with CC and constant covariance inflation of $\rho_t=0.8$ and the blue dashed line considered an increasing covariance inflation of $\rho_t=1-\frac{1}{\log(t+\exp(1)}$.}
\label{fig:err_param_obs_2d_vi_psylinear}
\end{figure}

In Figure~\ref{fig:err_param_obs_2d_vi_psylinear} we illustrate the error in the parameter space and observation space. The solutions computed by both methods with convex constraints converge to the KKT point $u_\ast^\tau$, whereas the EKI without constraints does not converge, since it does not take the constraints into account. Furthermore, we note that both methods with covariance inflation perform similarly well. In regard to the performance of the algorithms this is an important observation since the ODE \eqref{eqn:dyn_varinfl} get stiffer with higher value of covariance inflation and therefore also computationally slower. Since the experiment with $\rho_t=0.8$ is performing similarly well as with $\rho_t=1-\frac{1}{\log(t+\exp(1)}$ one should preferably implement a fixed value of covariance inflation.

\begin{figure}[t]
\centering
\captionsetup{width=.9\linewidth}
\includegraphics[scale=0.4]{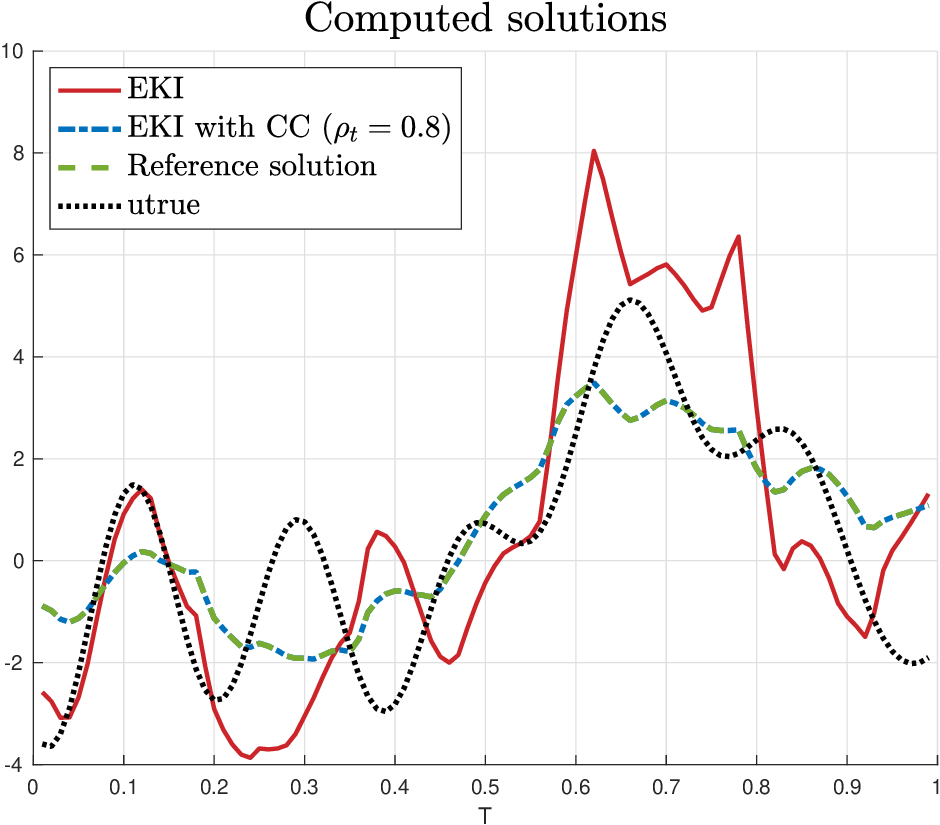}
\includegraphics[scale=0.4]{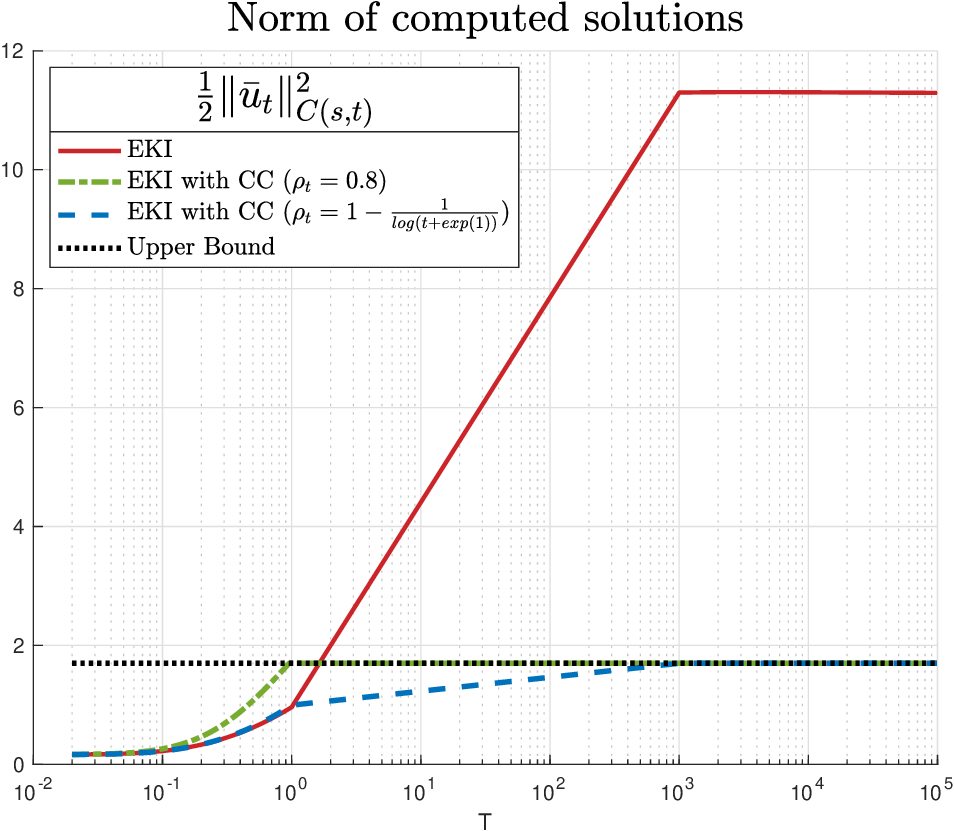}
\caption{Estimated computed solutions (left) and weighted norm of the of the computed solutions (right).} 
\label{fig:est_sol_1d_cons_vi}
\end{figure}
In the left image of Figure~\ref{fig:est_sol_1d_cons_vi} we illustrate the computed solutions. We only illustrate the solution computed with $\rho=0.8$, the dashed-green line, since it is very similar to the one for increasing covariance inflation. We can see that that it is visually identical to the KKT-point $u_\ast^\tau$, whereas the EKI solution, is has more discrepancies. This can also be seen in the right image, where we illustrate the weighted norm of the solutions. The dotted black line, illustrates the upper bound, i.e., $\frac{1}{2}\|f^\dagger\|_C(s,t)$. We can see that the norms of both solutions with CC (blue and green line) increase over time but after reaching the upper bound they stay below this bound, whereas the EKI does not satisfy the bound, which explains why the error in Figure~\ref{fig:err_param_obs_2d_vi_psylinear} does not decrease.

\subsection{2D-Darcy flow} \label{subsec:2d-darcy}

We consider the following elliptic Partial Differential Equation

\begin{equation} \label{pde:2ddarcy}
    \left\{
    \begin{aligned}
                -\nabla \cdot \left(\exp(u)\nabla p\right)  &= f,  \qquad  x\in D\\
            p&=0,   \qquad     x\in\partial D
    \end{aligned}
    \right.,
\end{equation}
where $D=\left(0,1\right)^2$ denotes the domain. Given observations of the solution $p\in H_0^1(D)\cap H^2(D):=\mathcal{V}$, our objective is to estimate the unknown diffusion coefficient $u^\dagger\in C^1(D)=X$. Additionally, we assume that the scalar field $f\in\mathbb{R}$ is known and define an observation operator for $K$ randomly chosen points in $X$ , i.e. $\mathcal{O}(p)=\left(p(x_1),...,p(x_K)\right)$.\\
The observations are represented by the equation
\[y=\mathcal{O}(p)+\eta\, .\]
The noise on our data, denoted by $\eta$ is assumed to be Gaussian, representing a realization of $\mathcal{N}(0,\Gamma)$, where $\Gamma=0.1^2 \mathrm{Id}_K$.\\
Thus, our inverse problem is given by
\[y=\mathcal{G}(u)+\eta\, ,\]
where $\mathcal{G}=\mathcal{O}\circ G$ and $G:X\rightarrow\mathbb{R}^K$ denotes the solution operator of the PDE \eqref{pde:2ddarcy}. We solve the PDE on a uniform mesh of the size $32\times 32$ using a FEM method with continuous, piece wise linear finite element basis functions. Consequently, the parameter space is of size $d=1024$.
As prior distribution we consider the random field
\[u(x,\omega)=\sum_{i=1}^s\lambda_i^{1/2}e_i(x)\xi_i(\omega)\, ,\]
where $\lambda_i=\left(\pi^2(k_j^2+l_j^2)+\tau^2\right)^{-\alpha}$ and $e_i(x)=\cos(\pi x_1 k_j)\cos(\pi x_2 l_j)$ with $\tau=0.01,\alpha=2,s=25,(k_j,l_j)_{j\in\{1,...,s\}}\in\{1,...,s\}^2$. The variables $\xi_i$ are i.i.d standard normal variables.\\
We select our underlying groundtruth by drawing $u^\dagger$ from the prior distribution. Furthermore, we consider Box-constraints (BC) in the form of upper and lower bounds on the solution, i.e. the solutions have to satisfy $(u_t^{(j)})_i\in\left[a,b\right]$, where $a=\min{u^\dagger}+0.3|\min{u^\dagger}|$ and $b=\max{u^\dagger}-0.3|\max{u^\dagger}|$ for all $j=1,...,J, t\geq0$ and components $i=1,...,d$.

Again we use MATLABs optimisation tool \verb+fmincon+ to compute the KKT point of \eqref{eqn:log_barrier_sys} and solve the ODE \eqref{eqn:dyn_cov_scalar} with \verb+ode45+. We compare again constant covariance inflation $\rho_t=0.8$ for all $t\geq0$ and increasing covariance inflation $\rho_t=1-\frac{1}{\log(t+\exp(1)}$. The particle size is set to $J=10$ and therefore, the approach seeks a solution in the subspace $u_0^\perp + \spann\{u_0^{(j)}-\bar{u}_0, j=1,...,J\}$.

We set a fixed penalty parameter $\tau=10^4$, a regularisation factor of $\beta=0.01$ and compute the solutions up until $T=10^6$.\\

\begin{figure}[t]
\centering
\captionsetup{width=.9\linewidth}
\includegraphics[scale=0.4]{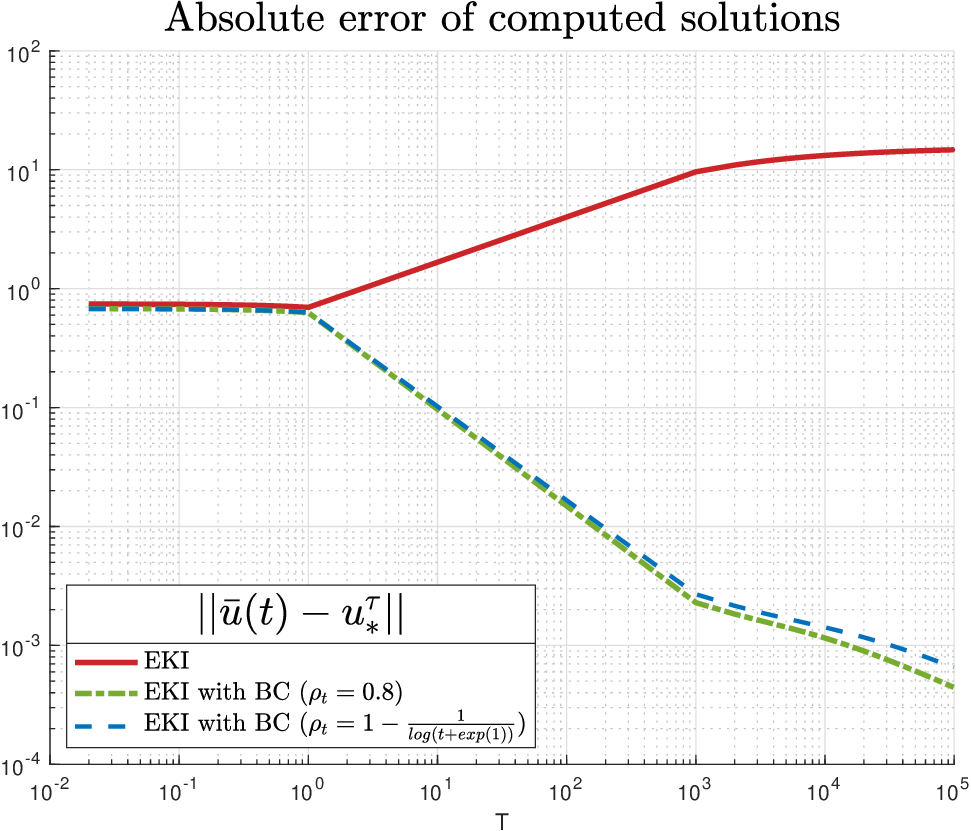}
\includegraphics[scale=0.4]{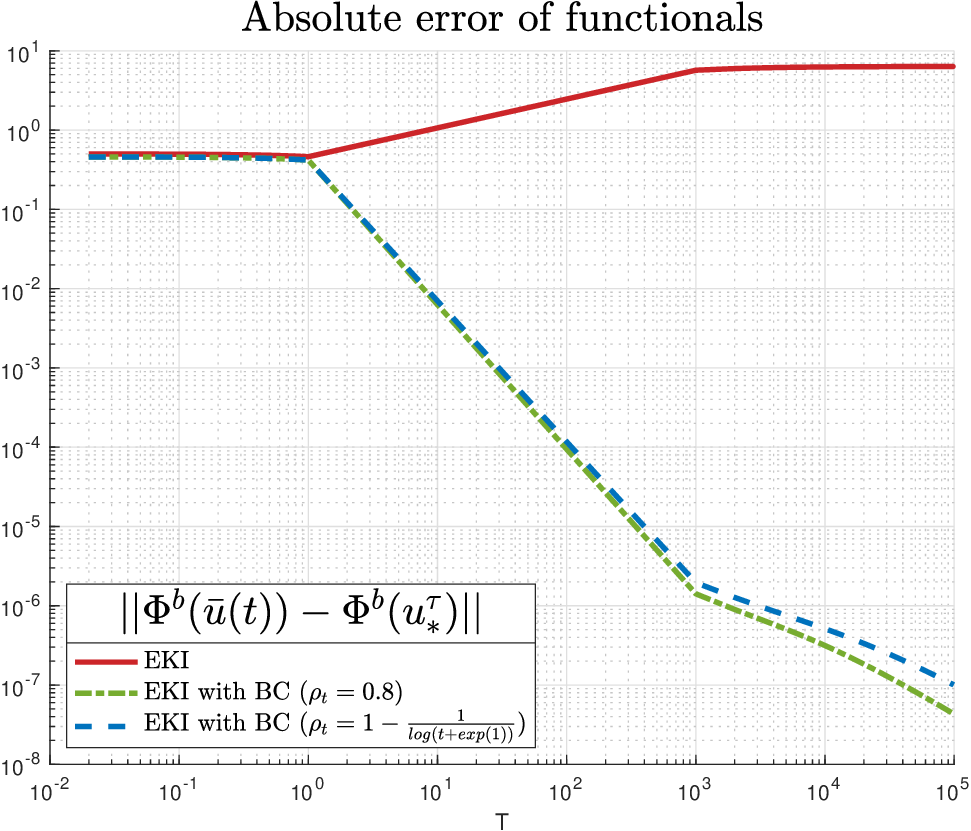}
\caption{Absolute error of computed solutions in the parameter (left) and observation space (right). The red line illustrates the EKI, the green line EKI with BC and constant covariance inflation of $\rho_t=0.8$ and the blue line uses $\rho_t=1-\frac{1}{\log(t+\exp(1)}$.}
\label{fig:err_param_obs_2d_nvi}
\end{figure}

In Figure~\ref{fig:err_param_obs_2d_nvi} we can see similarly to the experiment in Section~\ref{ex:psylinea} that the errors to the KKT point $u_\ast^\tau$ of \eqref{eqn:log_barrier_sys} of the methods with BC perform similarly well and both converge to $0$, whereas the original EKI does not converge. 

\begin{figure}[t]
\centering
\captionsetup{width=.9\linewidth}
\includegraphics[scale=0.7]{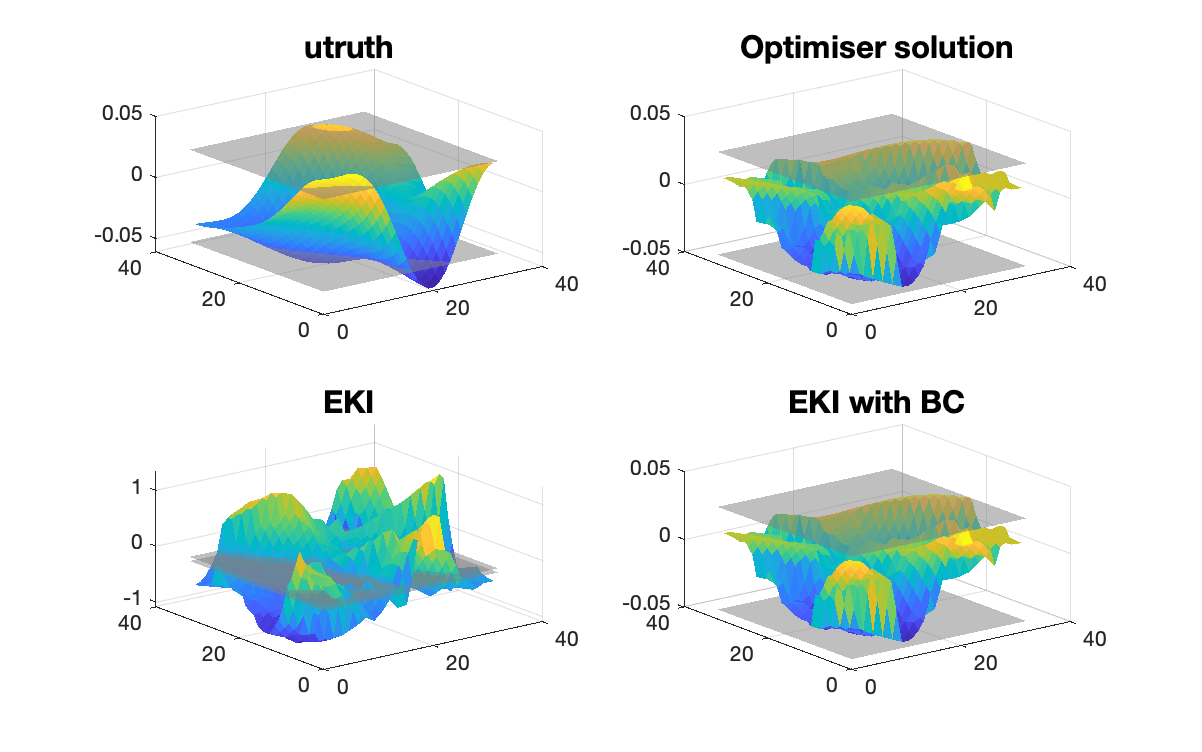}
\caption{Estimated solutions of the EKI (bottom left) and EKI with BC, where $\rho_t=1-\frac{1}{\log(t+\exp(1)}$ (bottom right) and optimiser (upper right). Upper left illustrates the true solution $u^\dagger$. The grey planes illustrate the upper and lower bound} 
\label{fig:est_sol_2d_vi_3d}
\end{figure}

Figure~\ref{fig:est_sol_2d_vi_3d} illustrates the estimated solutions of our methods. We can see that the EKI (bottom left) is outside the bounds, whereas our method with BC satisfies both upper and lower bound and looks very similar to the KKT-point obtained by MATLABs \verb+fmincon+ solver.

\subsubsection{Adaptive penalty parameter} \label{sec:adap_tau}
For our last experiment with adaptive penalty parameter we consider a mesh of the size $6\times 6$, since we have to compute the solutions for several different values of $\tau$, which is computationally expensive. Therefore, we have $d=36$. We take the same upper and lower bounds as above. Furthermore, we compute the solution of \eqref{eqn:dyn_varinfl_adaptau} until time $T=10^5$ and consider fixed covariance inflation of the magnitude $\rho_t=0.7$. We consider five different fixed penalty parameters $\tau\in\{1,10,100,1000,10000\}$ and one adaptive penalty parameter, $\tau(t)=t+1$ to compare this solution with the one for the highest fixed penalty parameter, $\tau=10000$. The remaining variables are the same as above.
\begin{figure}[t]
\centering
\captionsetup{width=.9\linewidth}
\includegraphics[scale=0.45]{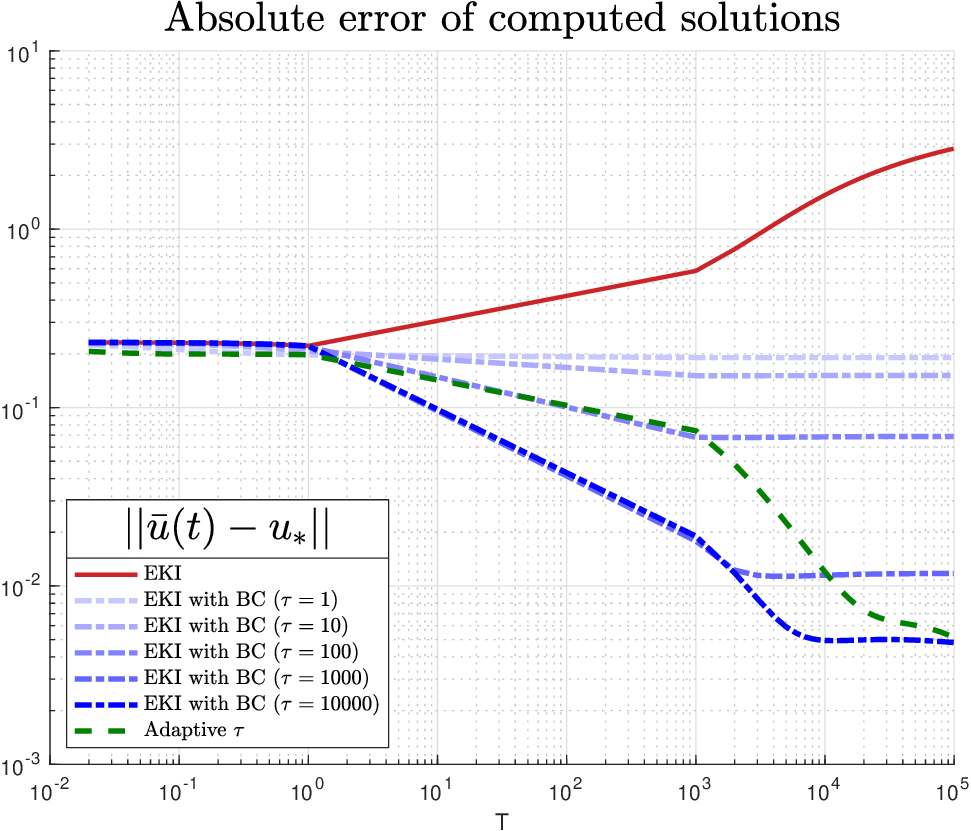}
\includegraphics[scale=0.45]{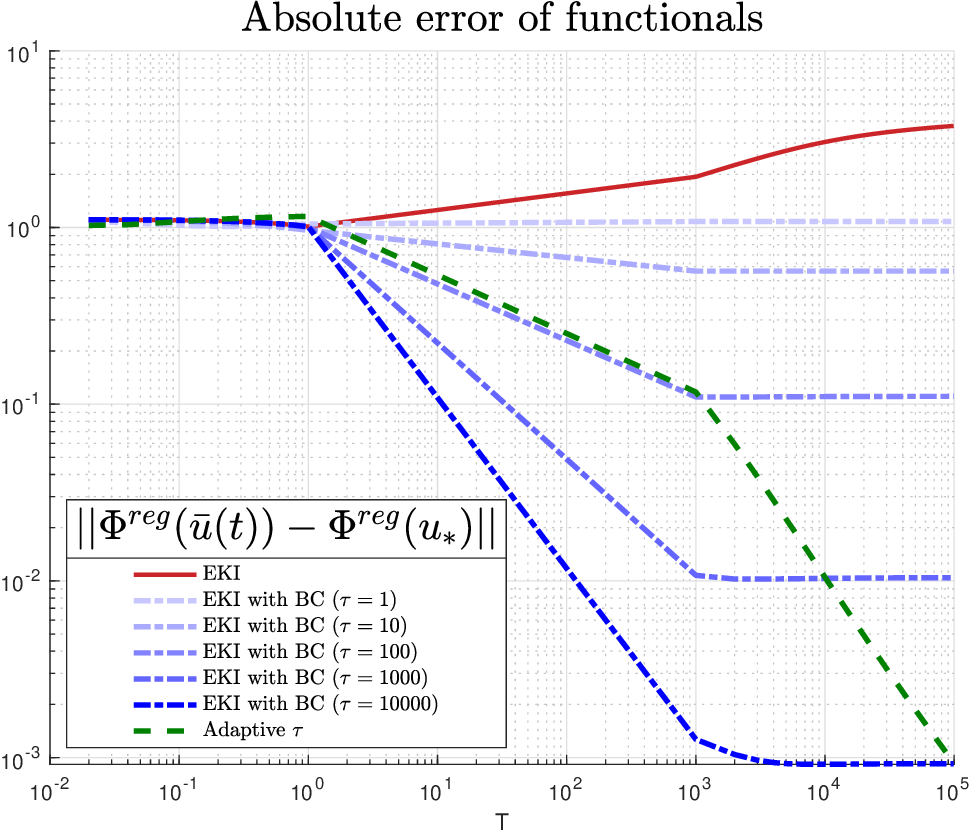}
\caption{Absolute error of computed solutions in the parameter (left) and observation space (right). The red line depicts the EKI. The dotted/dashed blue lines illustrates the EKI with BC for different fixed values of $\tau\in\{1,10,100,1000, 10000\}$. The dashed green line depicts the error of the the Box-constraints algorithm with adaptive increasing $\tau$.}
\label{fig:err_nonlinear_adaptau}
\end{figure}

In Figure~\ref{fig:err_nonlinear_adaptau} we illustrate the errors in the parameter space and observation space to the KKT-point $u_\ast$. We can see that as the fixed penalty parameter $\tau$ increases, the error decreases as described by equation $\eqref{eqn:error_true_kkt}$. Furthermore, for the adaptive choice of $\tau$ we can see that in the beginning the error vanishes slower then for fixed $\tau=10000$, but reaches the same error level at $T=10^5$.\\ This results indicates that the adaptive choice of penalty parameter is a suitable alternative, since it results in a less stiff ODE that needs to be solved. The computational time of the adaptive choice was $4$-times faster than the fixed penalty parameter.

\section{Conclusion} \label{sec:conclusion}
Based on a $\log$-barrier approach we introduced an adaptation of the EKI that allows the incorporation of convex inequality constraints for nonlinear forward problems. Through the inclusion of Tikhonov regularisation as well as covariance inflation we provided a convergence analysis for our method, which includes the quantification of the ensemble collapse as well as convergence as optimisation method. Our numerical experiments confirmed the results. In one experiment 
we consider a nonlinear forward problem given as linear problem with additional nonlinear perturbation. 
The magnitude of nonlinearity affects the amount of regularisation needed to obtain a strongly convex objective function and thus, satisfies our assumptions needed for Theorem~\ref{thm:main_conv}. In our second numerical experiment we considered the 2D-Darcy flow and showcased a method to address the problem of large penalty parameters that are needed in barrier-methods.\\
For future work we look to apply convex inequality constraints into subsampling approaches for the EKI \cite{Hanu_2023}. Moreover, it would be interesting to incoporporate the proposed approach to particle based sampling methods such as the ensemble Kalman sampler \cite{Garbuno2020}.


\section*{Acknowledgements}
MH is grateful for the support from MATH+ project EF1-19: Machine Learning Enhanced Filtering Methods for Inverse Problems, funded by the Deutsche Forschungsgemeinschaft (DFG, German Research
Foundation) under Germany's Excellence Strategy – The Berlin Mathematics
Research Center MATH+ (EXC-2046/1, project ID: 390685689).
The authors are very grateful to Claudia Schillings for the helpful discussions as well as proofreading the manuscript.

\bibliographystyle{abbrv}
\bibliography{main}

\appendix
\section{Appendix}
\subsection{Proofs of Lemmas}

Firstly, we show that strong convex functions fulfil the PL - inequality

\begin{lemma}\label{lemma:pl-phib}
    Let $\Phi:\mathbb{R}^{d}\rightarrow \mathbb{R}$ be strongly convex, then $\Phi$ satisfies the PL-inequality, i.e.
    $$\nu\|\nabla \Phi(x)\|^2\geq \Phi(x)-\Phi(x^*),$$
    for $\nu = \frac{1}{2\mu}>0$ and all $x\in \mathbb{R}^{d}$, where $x^*$ is a stationary point.
\end{lemma}

\begin{proof}
    Since $\Phi $ is strongly convex, there exists a $\mu>0$ such that 

    \[\Phi(y)\geq \Phi(x)+(y-x)\nabla\Phi(x)+\frac{\mu}{2}\|y-x\|^2,\quad \forall x,y\in\mathbb{R}^{d}\]

    which implies 
    \[\Phi(x)-\Phi(y)\leq \langle x-y,\nabla\Phi(x)\rangle-\frac{\mu}{2}\|x-y\|^2.\quad \forall x,y\in\mathbb{R}^{d}\]

    Note that $\|a-b\|^2=\|a^2\|-2\langle a,b\rangle+\|b\|^2$ for all $a,b\in\mathbb{R}^{d}$. Therefore, we have $\langle a,b\rangle-\frac{1}{2}\|a^2\|=\frac{1}{2}\|b^2\|-\frac{1}{2}\|a-b\|^2$. Setting $a=(x-y)\sqrt{\mu}$ and $b=\frac{1}{\sqrt{\mu}}\nabla \Phi(x)$, we obtain

    \begin{align*}
        \Phi(x)-\Phi(y)\leq&\frac{1}{2}\|\frac{1}{\sqrt{\mu}}\nabla\Phi(x)\|^2-\frac{1}{2}\|(x-y)\sqrt{\mu}-\frac{1}{\sqrt{\mu}}\nabla \Phi(x)\|^2\\
        \leq&\frac{1}{2\mu}\|\nabla\Phi(x)\|^2.
    \end{align*}
    Hence, $\Phi$ satisfies the PL-inequality with constant $\frac{1}{2\mu}$.
\end{proof}

The next statement considers our applied inverse problem based on a linear forward model with nonlinear perturbation. It provides sufficient conditions for the Tikhonov regularized loss function $\Phi^{\scaleto{\mathrm{reg}}{5pt}}$ to be strongly convex.
\begin{lemma} \label{lemma:ex_convex}
    Consider the potential \eqref{eqn:regul_pot} where $G(u)=Au+\varepsilon\left[\sin(u_1),...,\sin(u_K)\right]^T$, where $A\in\mathbb{R}^{K\times K}$ and $\varepsilon>0$. Further, assume that $\sup_{u\in\Omega} \|u\|\le B$ and let $A_{\max} = \max_{i,j}|A_{i,j}|$, $y_{\max} = \max_{i}|y_i|$ and $\lambda > ((4+B)a_{\max} + y_{\max}) \varepsilon + \varepsilon^2 $, then the Tikhonov regularised loss function 
    \[\Phi^{\scaleto{\mathrm{reg}}{5pt}}(u) = \frac12\|G(u)-y\|^2 + \frac{\lambda}2 \|u\|^2\]
    is strongly convex and $L$-smooth.
\end{lemma}

\begin{proof}
    The Jacobian $J_G$ of $G$ is given by
    \[J_G(x)=\left(A^T+\varepsilon diag(\cos(x_1),...,\cos(x_K)\right).\]
    Hence, the $i$-th row and $j$-th component of the row have the following structure
    \begin{align*}
        (J_G(x))_{i,:}&= A^T_{i,:}+\varepsilon\left[0,...,\cos(x_i),...,0\right]\\
        (J_G(x))_{i,j}&= A_{i,j}+\varepsilon \cos(x_i)\delta_{ij},
    \end{align*}
    where $\delta_{i,j}$ denotes the Kronecker symbol. Hence, we have
    \[\left(\frac{\partial\Phi}{\partial x}\right)_i=\sum_{j=1}^K\left(A_{j,i}+\varepsilon \cos(x_i)\delta_{ij}\right)\left(A_{j,:}x+\varepsilon\sin(x_j)-y_j\right).\]
    And therefore, we obtain for the entries of the Hessian
    \begin{align*}
        \frac{\partial^2\Phi}{\partial x_i\partial x_l}=&\sum_{j=1}^K\left(-\varepsilon \sin(x_i)\delta_{ij}\delta_{il}\right)\left(A_{j,:}x+\varepsilon\sin(x_j)-y_j\right)\\
        +&\sum_{j=1}^K\left(A_{j,i}+\varepsilon \cos(x_i)\delta_{ij}\right)\left(A_{j,l}+\varepsilon\cos(x_j)\delta_{jl}\right)+\lambda\\
        =&\left(-\varepsilon \sin(x_i)\delta_{il}\right)\left(A_{i,:}x+\varepsilon\sin(x_i)-y_i\right)\\
        +&\sum_{j=1}^K A_{j,i} A_{j,l}+\sum_{j=1}^K A_{j,i}\varepsilon\cos(x_j)\delta_{jl}\\
        +&\sum_{j=1}^K \varepsilon \cos(x_i)\delta_{ij}A_{j,l}+\varepsilon^2\cos(x_i)^2\delta_{il}+\lambda\, .
    \end{align*}
    Thus, the entries of the Hesse-Matrix are given by   
\[\left(H_{\Phi^{\scaleto{\mathrm{reg}}{5pt}}})\right)_{i,l}=\begin{cases}
			\sum_{j=1}^K A_{j,i}^2-\varepsilon\sin(x_i)\left(A_{i,:}x+\varepsilon\sin(x_i)-y_i\right)+ 2A_{i,i}\varepsilon\cos(x_i)+\varepsilon^2\cos(x_i)^2+\lambda, & \text{if $i=l$},\\
            \sum_{j=1}^K A_{j,i} A_{j,l}+A_{l,i}\varepsilon\cos(x_l)+ A_{i,l}\varepsilon \cos(x_i), & \text{if $\neq l$.}
		 \end{cases}
\]
We can split the Hessian into two parts:
\[\left(H_{\Phi^{\scaleto{\mathrm{reg}}{5pt}}})\right)=A^TA+M_{\varepsilon,\lambda}\, ,\]
where 
\[M_{\varepsilon,\lambda}=\begin{cases}
			-\varepsilon\sin(x_i)\left(A_{i,:}x+\varepsilon\sin(x_i)-y_i\right)+ 2A_{i,i}\varepsilon\cos(x_i)+\varepsilon^2\cos(x_i)^2+\lambda, & \text{if $i=l$,}\\
           \varepsilon\cos(x_l)+ A_{i,l}\varepsilon \cos(x_i), & \text{if $\neq l$.}
		 \end{cases}
\]
Since $A^TA$ is positive semi-definite we continue to show that $M_{\varepsilon,\lambda}$ is positive definite.
We start by lower bounding the diagonal elements of $\left(M_{\varepsilon,\lambda}\right)$ by
\[(M_{\varepsilon,\lambda})_{i,i} \ge - \varepsilon (A_{\max} B + \varepsilon + y_{\max}) - 2 A_{i,i} \varepsilon + \lambda\,, \]
where $A_{\max} = \max_{i,j}\ |A_{i,j}|$, $y_{\max}= \max_{i}|y_i|$ and $B = \sup_{x\in\Omega} \|x\|<\infty$. 
On the other side, we can bound the off-diagonal elements by
\[ \sum_{l=1}^K|(M_{\varepsilon,\lambda})_{i,l}| \le  2A_{\max}\varepsilon\,.\]
Hence, by assumption we have for all $i\in\{1,...,K\}$
\[\left(M_{\varepsilon,\lambda}\right)_{i,i}\geq - \varepsilon (A_{\max} B + \varepsilon + y_{\max}) - 2 A_{i,i} \varepsilon + \lambda > 2A_{\max}\varepsilon \geq \sum_{i\neq j} |\left(M_{\varepsilon,\lambda}\right)_{i,j}|\, .\]
Therefore, $M_{\varepsilon,\lambda}$ is positive definite and by that also $H_{\Phi^{\scaleto{\mathrm{reg}}{5pt}}}$. Furthermore, we can upper bound the diagonal elements of $M_{\varepsilon,\lambda}$ by
\[(M_{\varepsilon,\lambda})_{i,i} \le \varepsilon (A_{\max} B + \varepsilon + y_{\max}) + 2 A_{i,i} \varepsilon +\varepsilon^2+ \lambda\,. \]
Hence, all entries of $H_{\Phi^{\scaleto{\mathrm{reg}}{5pt}}}$ are bounded implying the $L$-smothness of $\Phi^{\scaleto{\mathrm{reg}}{5pt}}.$
\end{proof}
The next lemma provides the solution of the ODE needed to derive the lower bound on the ensemble collapse in Lemma~\ref{lemma:ens_col_lower}.
\begin{lemma} \label{alem:ode_lower_bound}
Let $a,b,c>0$ then the solution of the ODE
\begin{equation}\label{eq:app:ODE}
    \frac{{\mathrm d}x(t)}{{\mathrm d}t} = - \frac{a}{t+b} x(t) - c x(t)^2
\end{equation}
is given by
\begin{equation} \label{eqn:sol_ode}
    x(t) = \frac{1-a}{c(t+b) + (1-a) k_1 (t+b)^{a}},\quad x(0) = x_0>0\,,
\end{equation}
where $k_1$ is chosen such that the initial condition is satisfied.
\end{lemma}
\begin{proof}
    The proof follows the ideas of Lemma A.3 in \cite{Tong_2023}. We show that that \eqref{eqn:sol_ode} solves the ODE:
    \begin{align*}
        \frac{{\mathrm d}x(t)}{{\mathrm d}t} =&-\frac{(1-a)(c+a(1-a)k_1(t+b)^{a-1})}{\left(c(t+b) + (1-a) k_1 (t+b)^{a}\right)^2}\\
        =&\left(-\frac{c+a(1-a)k_1(t+b)^{a-1}}{c(t+b) + (1-a) k_1 (t+b)^{a}}\right)\left(\frac{1-a}{c(t+b) + (1-a) k_1 (t+b)^{a}}\right)\\
        =&\left(-\frac{c+ca-ca+a(1-a)k_1(t+b)^{a-1}}{c(t+b) + (1-a) k_1 (t+b)^{a}}\right)x(t)\\
        =&\left(-\frac{ca+a(1-a)k_1(t+b)^{a-1}+c(1-a)}{c(t+b) + (1-a) k_1 (t+b)^{a}}\right)x(t)\\
        =&\left(-\frac{a(t+b)\left(c+(1-a)k_1(t+b)^{a-1}\right)}{(t+b)(c(t+b) + (1-a) k_1 (t+b)^{a})}-cx(t)\right)x(t)\\
        =&\left(-\frac{a}{t+b}-cx(t)\right)x(t)\,,
    \end{align*}
    which is in the needed form of \eqref{eq:app:ODE}.
\end{proof}

\subsection{Proof of the feasibility and unique existence of solution}\label{app:A2}
\begin{proof}[Proof of Proposition~\ref{lemma:ens_in_box}]
    Without loss of generality we assume that $m=1$, i.e.~there is only one inequality constrain $h(u)\le 0$. The evolution of $h(\bar u_t)$ is given by
    \[\frac{{\mathrm d}h(\bar u_t)}{{\mathrm d}t} = \langle \nabla h(\bar u_t),\frac{{\mathrm d}\bar u_t}{{\mathrm d}t}\rangle = \langle \nabla h(\bar u_t), v(u_t)\rangle + \frac{1}{\tau}\frac{1}{h(\bar u_t)} \langle \nabla h(\bar u_t), \widehat C(u_t) \nabla h(\bar u_t) \rangle\, , \]
    where $v$ is defined in \eqref{eq:drift_force}. Let $T>0$, then for any solution $(\bar u_t)_{t\in[0,T]}$, $h(\bar u_T)$ can be represented by
    \[h(\bar u_T) = h(\bar u_0) + \int_0^T\left(\langle \nabla h(\bar u_t), v(u_t)\rangle + \frac{1}{\tau}\frac{1}{h(\bar u_t)} \langle \nabla h(\bar u_t), \widehat C(u_t) \nabla h(\bar u_t) \rangle\right)\, {\mathrm{d}}t\,.\]
    Suppose there exists $T_1>0$ such that $h(\bar u_{T_1})>0$. Since we initialize with $h(\bar u_{0})<0$ by continuity of solutions and continuity of $h$, it follows that there exists $t_{\min}\in (0,T_1)$ such that $h(\bar u_{t_{\min}}) = 0$ and $h(\bar u_{t})<0$ for all $t<t_{\min}$. Moreover, again by continuity it follows that there must exists $\delta>0$ such that  
    \begin{equation} \label{eq:contr_claim}
    \langle \nabla h(\bar u_t), v(u_t)\rangle + \frac{1}{\tau}\frac{1}{h(\bar u_t)} \langle \nabla h(\bar u_t), \widehat C(u_t) \nabla h(\bar u_t) \rangle>0
    \end{equation}
    for all $t\in (t_{\min}-\delta,t_{\min})$ and therefore, also $\nabla h(\bar u_t) \neq 0$ for all $t\in (t_{\min}-\delta,t_{\min})$. Remember, by Assumption~\ref{ass:inequality_constr} we  also have that $\nabla h(\bar u_{t_{\min}}) \neq 0$, see Remark~\ref{rem:non_degenerate_grad}.  
    Next, for $s\in[0,t_{\min})$ we consider
    \[ v(u_s) \le \|\widehat{C}^{u,G}_s\|_{F}\|\Gamma^{-1}\|_{F}\left(\|\bar{G}_s\|+\|y\|\right)+\|\widehat{C}(u_s)\|_F\|C_0^{-1}\|_F\|\bar{u}_s\| \]
    We can approximate the covariance $\widehat{C}^{u,G}_s$ by the following 
    \begin{align*}
        \|\widehat{C}^{u,G}_s\|_F^2\leq \left(\frac{1}{J}\sum_{j=1}^J \|e_s^{(j)}\|^2 \right) \left( \frac1J \sum_{j=1}^J \| G(u_s^{(j)}) -\bar G_s \|^2 \right)
    \end{align*}
    using a same argumentation as in proof of Lemma~\ref{lemma:ens_col_lower}. We approximate the first term by
    \[\frac{1}{J}\sum_{j=1}^J \|e_s^{(j)}\|^2 = V_e(s) \le V_e(0)\]
    and the second factor by
    \begin{align*}
        \frac{1}{J}\sum_{j=1}^J\|G(u^{(j)}_s)-\bar{G}\|^2
        \leq c_{lip}^2 V_e(s)
        \le c_{lip}^2 V_e(0),
    \end{align*}
    where we used Lemma~\ref{lemma:bound_enskol}. 
    Hence, we have
    $\|v({u}_s)\|\leq  B$
    for all $s\in\left[0,t_{min}\right]$ and some $B>0$. Similarly, using that $\|\nabla h(u)\|\le C(R)$, we have that $|\langle \nabla h(\bar u_s), v(u_s)\rangle | \le B$ for $s\in\left[0,t_{min}\right]$ and some constant $B$.
    Finally, we observe that  
    \[ \frac{1}{\tau}\frac{1}{h(\bar u_s)} \langle \nabla h(\bar u_s), \widehat C(u_s) \nabla h(\bar u_s) \rangle < 0 \]
    for all $s\in\left[0,t_{min}\right)$, since $\langle \nabla h(\bar u_s), \widehat C(u_s) \nabla h(\bar u_s) \rangle>0$ using Lemma~\ref{lemma:ens_col_lower}. 
    Therefore, using 
    \[\lim_{\varepsilon\to0}\frac{1}{h(\bar u_{t_{\min}-\varepsilon})} = -\infty \]
    since $h(\bar u_{t_{\min}-\varepsilon})\nearrow 0$ for $\varepsilon\to0$, and $\nabla h(\bar u_t) \neq 0$ for all $t\in(t_{\min}-\delta,t_{\min}]$, we obtain 
    \[ \lim_{\varepsilon\to0} \langle \nabla h(\bar u_{t_{\min}-\varepsilon}), v(u_{t_{\min}-\varepsilon})\rangle + \frac{1}{\tau}\frac{1}{h(\bar u_{t_{\min}-\varepsilon})} \langle \nabla h(\bar u_{t_{\min}-\varepsilon}), \widehat C(u_{t_{\min}-\varepsilon}) \nabla h(\bar u_{t_{\min}-\varepsilon})\rangle  = -\infty \]
    which is in contradiction to \eqref{eq:contr_claim}. For the second part, we know by local Lipschitz-continuity that locally there exists a unique solution of \eqref{eqn:dyn_varinfl}. Hence, as the local solution remains in $\Omega$ it is also a global solution.
\end{proof}

\end{document}